\documentclass{amsart}
\usepackage{pstricks}
\usepackage{amssymb}
\usepackage{amsfonts}
\usepackage{epsfig}
\usepackage[cp850]{inputenc}
\usepackage[english]{babel}
\usepackage{ifthen}
\usepackage{graphicx}
\usepackage{float}
\usepackage{amstext}
\usepackage{syntonly}
\usepackage{amsthm}
\usepackage{amsmath}
\usepackage{amssymb}
\newtheorem{theorem}{Theorem}

\newtheorem{corollary}[theorem]{Corollary}

\newtheorem{definition}[theorem]{Definition}

\newtheorem{lemma}[theorem]{Lemma}

\newtheorem{notation}[theorem]{Notation}

\newtheorem{proposition}[theorem]{Proposition}
\newtheorem{remark}[theorem]{Remark}

\def\@bibliosize{\small}
\def\@historysize{\small}
\def\@keywordsize{\small}
\def\@overaddressskip{2pt}
\def\@titlesize{\Large\bfseries}
\def\@authorsize{\large}
\def\@keywordheading{{\it Key words: \ }}
\def\@addressstyle{\small\itshape}
\def\@captionsize{\small}
\def\@tablecaptionsize{\@captionsize}
\def\@figurecaptionsize{\@captionsize}
\def\@tablesize{\small}
\def\@keywordwidth{.8\textwidth}
\def\@abstractwidth{.8\textwidth}
\def\blfootnote{\xdef\@thefnmark{}\@footnotetext}

\begin{document}

\title[the completion problem of quasi-uniform spaces]
{A solution to the completion problem of quasi-uniform spaces}

\author{Athanasios Andrikopoulos \ \ John Stabakis}

\begin{abstract} We give a new completion for the quasi-uniform spaces. We call the
whole procedure {\it $\tau$-completion} and the new space {\it
$\tau$-complement of the given}. The basic result is that every
$T_{_0}$ quasi-uniform space has a $\tau$-completion. The
$\tau$-complement has some \textquotedblleft
crucial\textquotedblright properties, for instance, it coincides
with the classical one in the case of uniform space or it extends
the {\it Doitcinov's completion for the quiet spaces}. We use nets
and from one point of view the technique of the construction may be
considered as a combination of the {\it Mac Neille's cut} and of the
completion of partially ordered sets via {\it directed subsets}.
\end{abstract}

\maketitle

{\it {\rm 2000 AMS Classification}: 54E25, 54E35, 54E55.}

{\it Key words: {\rm completion of a quasi-uniform space,
quiet spaces, $D$-completion.}}

\section{Introduction}

It is our main purpose in this paper to give a standard construction
for the completion of any $T_{_0}$ quasi-uniform space.
Cs\'{a}sz\'{a}r was the first who developed a theory of completion for quasi uniform spaces \cite{csa}.
Since then, the problem has been
approached by several authors, but, up to now, none of
the solutions proposed is able to give a satisfying completion theory for
all quasi-uniform spaces.
We must add that in two rather recent papers,
[Bosangue and all, 1998] and [Kunzi and all, 2002], it is attempted to be given a
definite resolution of the subject. In the first of them, starting
from a Lawvere's idea \cite{law} on a generalization of metric and
ordered spaces, the authors give a completion of quasi-uniform spaces,
rather according to the completion of some families of quasi-uniform spaces than to the completion of an ordered space.
In the second, they give the
completion of some families of generalized metric spaces considering it as the {\it Yoneda
embedding}. However, although for some cases of quasi-uniform spaces a {\it standard
construction} is possible, there is not a general method for any
quasi-uniform space. Even simple examples,
as, for instance, the topological ordered spaces, are not in general covered by the
current constructions.

It is not difficult for one to understand that the problem is not an
easy task. \textquotedblleft {\it The problem of completing of quasi-uniform spaces
is not a trivial one. The analogy with the uniform spaces provides a
too small profit}\textquotedblright says Doitchinov in \cite{doi3}.
S\"{u}nderhauf in \cite{sun1}: \textquotedblleft {\it there are
various attempts to define a notion of completeness and completion,
but none of these is able to handle all spaces in a satisfying
manner}\textquotedblright.

And later, the same author in \cite{sun2}: \textquotedblleft {\it as
soon as the notion of Cauchy sequences} ({\it and filters}), {\it limits and
completeness come into play, the situation becomes rather
chaotic}\textquotedblright. And others say such things in other papers.

From one point of view the problem of completing quasi-uniform spaces is due to the
plethora of cases which one has to confront. This plethora is due to
the fact that in a quasi-uniform spaces a \textquotedblleft {\it natural}\textquotedblright
Cauchy system does not converge in an also \textquotedblleft {\it natural
way}\textquotedblright. Consider, for instance, the real line
$\mathbb{R}$ with a topology $\tau$, whose a subbase consists of the
intervals $(\leftarrow,x]$, for any $x\in \mathbb{R}$. If $\mathcal{U}$ is a quasi-uniformity compatible
with $\tau$, then, for any
$a\in\mathbb{R}$, any sequence $(x_{_n})_{_{n \in \mathbb{N}}}$ with
$x_{_n}<a$, ($\mathbb{N}$ the set of
natural numbers), converges to $a$.

Such facts explaining why, amongst the many proposals,
Doitchinov's $D$-completion (firstly in \cite{doi1}, \cite{doi2},
etc.) of the so-called \textquotedblleft{\it quiet
spaces}\textquotedblright was particularly promising. He introduced
the idea of considering with every {\it Cauchy net}, a second net -
the so called {\it conet of the former} - which couple of nets
restricts the area where the \textquotedblleft new point
\textquotedblright must stand and so it is possible for one to give
a completion, the so called {\it $D$-completion}, for the so-called
{\it quiet spaces}. Unfortunately, although too many current
examples of quasi-uniform spaces have $D$-completion, that was not a resolution:
\textquotedblleft {\it the quiet spaces constitute a very small
class of spaces}\textquotedblright (cf. \cite{red}). Fletcher, in one his
reviewing, notes: \textquotedblleft {\it D.Doitchinov
showed that it is impossible to give a satisfactory theory of
completion for the class of quasi-uniform spaces and introduced the class of quiet
quasi-uniform spaces, a class comprising all uniform spaces for which a satisfactory
theory of completion exists}\textquotedblright.
\par
\smallskip
\par
The usual meaning of completeness is that every Cauchy system
(sequence or net or filter), whatever such a system means,
converges. So the definition of the Cauchy system is the first task.
But such a definition is not always a natural one. In some cases the
converging nets are not Cauchy nets, in some others they are Cauchy,
but cease to stand as Cauchy when we withdraw the limit points from
the space, in a third the completion has only one \textquotedblleft
{\it new point}\textquotedblright very artificially constructed and,
finally, we \textquotedblleft {\it complete}\textquotedblright with
\textquotedblleft {\it new points}\textquotedblright an already complete
space.
 In this paper, we introduce a new notion, the {\it
$\tau$-cut} (from the Greek word $\tau o \mu {\eta}$ $=$ cut). More
precisely, we make use of pairs \textquotedblleft
net-conet\textquotedblright
 as in the $D$-completion, here as
$\tau$-{\it net} and $\tau$-{\it conet}. The $\tau_{_p}$-nets are nets whose origin
is due to Stoltenberg \cite{sto}. For more details in the subject
see \cite{ku}.

Our basic conclusion is the Theorem \ref{a26} which states that every
$T_{_0}$ quasi-uniform spaces has a $\tau$-completion, that is a completion via the
$\tau$-cuts. The meaning of the $\tau$-Cauchy net of the
$\tau$-completion is given in the Definitions \ref{a8} and \ref{a9}.
As we have already said there exists a large number of
completions,
among them the Deak's and the ones of Smyth's and
Doitchinov's (there is a very large bibliography of Deak last papers; the paper \cite{dea} is among of them).

So, the paper is organized as follows: in the paragraphs 2 we give
the basic definitions and structures for the $\tau$-completion;
paragraph 3 is referred to the definition of a quasi-uniformity in
the $\tau$-completion and in the fourth paragraph we prove the
$\tau$-completeness of the new space we have constructed.

\section{The $\tau$-cut}

Let us recall that a quasi-uniformity on a (nonempty) set $X$ is a
filter $\mathcal{U}$ on $X\times X$
which satisfies: (i) $\Delta(X)=\{(x,x)\vert x\in X\}\subseteq U$
for each $U\in \mathcal{U}$
and (ii) given $U\in \mathcal{U}$
there exists $V\in \mathcal{U}$ such that
$V\circ
V\subseteq U$. The elements of the filter
$\mathcal{U}$
are called {\it entourages}.
The pair $(X,\mathcal{U})$ is called a {\it quasi-uniform
space}. If $\mathcal{U}$ is a quasi-uniformity on a set $X$, then ${\mathcal{U}}^{-1}=\{U^{-1}\vert U\in \mathcal{U}\}$ is also a
quasi-uniformity on $X$ called the {\it conjugate} of $\mathcal{U}$.
Given a quasi-uniformity $\mathcal{U}$ on $X$, $\mathcal{U}^{\star}=\mathcal{U}\bigvee \mathcal{U}^{-1}$ will denote the coarsest uniformity on $X$
which is finer than $\mathcal{U}$. If $U\in \mathcal{U}$, the entourage
$U\cap U^{-1}$ of $\mathcal{U}^{\star}$ will be denoted by $U^{\star}$.
Every quasi-uniformity $\mathcal{U}$ on $X$ generates a topology $\tau(\mathcal{U})$.
A neighborhood base for each point $x\in X$ is given by $\{U(x)\vert U\in \mathcal{U}\}$ where
$U(x)=\{y\in X\vert (x,y)\in U\}$.
\par
It is well known that there is a base for ${\mathcal{U}}$, which we always
denote by ${\mathcal{U}}_{_0}$, consisting of all $U\in {\mathcal{U}}$ that are
$\tau({\mathcal{U}}^{-1})\times \tau({\mathcal{U}})$-open in $X\times X$ (cf.
\cite[page 8]{FL}).
\par
According to D.Doitchinov \cite[Definition 1]{doi2}, a net
$(y_{_\beta})_{_{\beta\in B}}$ is called {\it a conet} of the net
$(x_a)_{a\in A}$, if for any $U\in {\mathcal{U}}$ there are $a_{_U} \in
A$ and $\beta _{_U} \in B$ such that $(y_\beta, x_a)\in U$ whenever
$a\geq a_{_U}$ and $\beta \geq \beta_{_U}$. In this case we write
$(y_{_\beta},x_{_a})\longrightarrow 0$ (see fig.1).

\pspicture(0.5,0.5)(1.8,1.8)

\Rput[ul](.3,1){$(x_{_a})_{_{a\in A}}$}

\psline[linestyle=dotted]{->}(.1,.5)(6.2,.5)

\psline[linestyle=dotted]{<-}(6.3,.5)(12,.5)

\Rput[ul](10.5,1){$(y_{_\beta})_{_{\beta\in B}}$}

\psline{<-}(3.3,.2)(6,.2)

\Rput[ul](5.95,.2){${\rm U}$}

\psline{->}(6.5,.2)(9,.2)

\qdisk(9,0.5){2pt}

\qdisk(3.25,0.5){2pt}

\Rput[ul](2.8,.8){$x_{_{a_{_U}}}$}

\Rput[ul](8.8,.8){$y_{_{\beta_{_U}}}$}

\endpspicture
\par\bigskip
\begin{center}
Figure 1
\end{center}
\par\smallskip\par

We
preserve the Doitchinov's duality of net-conet and we will make use
of a notion given by Kelly (\cite[page 75]{kel}) under the terminology of
$p$-sequence and $q$-sequence referring to the two directions of a
net, and by Stoltenberg (\cite[page 229]{sto}) who worked with only the one
direction.

\begin{definition}\label{a1}{\rm A net $(x_{_a})_{_{a\in A}}$ in a quasi-uniform space $(X,{\mathcal{U}})$
is called $\tau_{_p}$-{\it net} (resp. $\tau_{_q}$-{\it net}) {\it for} $W$, if there is an $a_{_W}\in A$ such that
$(x_{_{a^{\prime}}},x_{_a})\in
W$ (resp.
$(x_{_a},x_{_{a^{\prime}}})\in
W$)
for each $a^{\prime}\geq a\geq a_{_W}$.
The net $(x_{_a})_{_{a\in A}}$ is called $\tau_{_p}$-{\it net} (resp. $\tau_{_q}$-{\it net}), if for each
$W\in {\mathcal{U}}$, it is a
$\tau_{_p}$-{\it net} (resp. $\tau_{_q}$-{\it net}) for $W$.
We call $x_{a_{_W}}$
{\it extreme point for} $W$ of the $\tau_{_p}$-net (resp. $\tau_{_q}$-net) and the
$a_{_W}$ {\it extreme index for} $V$
(see fig. 2).}
\end{definition}

\par
We will make use of the phrase \textquotedblleft {\it final segment
of a $\tau_{_p}$-net}\textquotedblright: we mean, the
set of the elements of the $\tau_{_p}$-net from one  point of it onwards
({\it until the end}). We dually define the \textquotedblleft{\it
{final segment of a $\tau_{_q}$-net}\textquotedblright}.
\par
We also say for a net $(x_a)_{a\in A}$, $t \in X$ and $U\in
{\mathcal{U}}$:
\par
\begin{center}{ \textquotedblleft {\it finally} $(t, (x_a )_a ) \in U $
\textquotedblright or, in symbols, \textquotedblleft
$\tau.(t,(x_a)_a)\in U$\textquotedblright,}
\end{center}
if $(t,x_a)\in U$ for all the points $x_a$ of a final segment of
$(x_a)_{a\in A}$.

\pspicture(0.5,0.5)(1.8,1.8)

\Rput[ul](2,1){$(x_{_a})_{_{a\in A}}$}

\Rput[ul](.1,.5){$\tau_{_p}$-net}

\psline[linestyle=dotted]{->}(1.8,.5)(10,.5)

\psline{<-}(4.3,.1)(6.9,.1)

\Rput[ul](6.8,.1){$W(x_{_{a^{\prime}}}$)}

\psline{->}(8.2,.1)(8.9,.1)

\qdisk(4.25,0.5){2pt}

\Rput[ul](3.8,.8){$x_{_{a_{_W}}}$}

\qdisk(5.3,.5){2pt}

\Rput[ul](5,.8){$x_{_{a}}$}

\qdisk(7.3,.5){2pt}

\Rput[ul](6.9,.8){$x_{_{a^{\prime}}}$}

\Rput[ul](6.5,-.58){${\rm A}$}
\psline{-}(4.3,-.5)(9,-1.2)

\psline{-}(4.3,-1.5)(9,-1.2)

\qdisk(4.3,-1){1pt}
\Rput[ul](4,-.8){$a_{_W}$}

\qdisk(5.3,-1.3){1pt}

\Rput[ul](5,-1.1){$a$}

\qdisk(6.4,-1.2){1pt}

\Rput[ul](6.2,-1){$a^{\prime}$}

\endpspicture
\par\bigskip\bigskip
\begin{center}

\end{center}
\par\bigskip\smallskip\par

\pspicture(0.5,0.5)(1.8,1.8)

\Rput[ul](8.1,1){$(x_{_a})_{_{a\in A}}$}

\psline[linestyle=dotted]{<-}(1.8,.5)(10,.5)

\Rput[ul](.1,.5){$\tau_{_q}$-net}

\Rput[ul](6.8,.8){$x_{_{a_{_W}}}$}

\qdisk(4.3,.5){2pt}

\Rput[ul](4.1,.8){$x_{_{a}}$}

\qdisk(5.4,.5){2pt}

\Rput[ul](5.2,.8){$x_{_{a^{\prime}}}$}

\psline{<-}(3.3,.1)(3.7,.1)

\Rput[ul](3.6,.0){$W(x_{_a})$}

\psline{->}(4.9,.1)(7.1,.1)

\qdisk(7.05,0.5){2pt}

\endpspicture
\par\bigskip\bigskip\bigskip
\begin{center}
Figure 2
\end{center}
\par\bigskip\smallskip\par

\par
We give a similar meaning in the notations:
$\tau.((y_{_\beta})_{_\beta},t)\in U$ and in the same way
$\tau.((y_{_{\beta}})_{_{\beta}},(x_a)_a)\in
U$ (see fig. 3).

\pspicture(0.5,0.5)(1.8,1.8)

\Rput[ul](0.1,1){$(x_{_a})_{_{a\in A}}$}

\psline[linestyle=dotted]{->}(0.2,.5)(5,.5)

\psline{<-}(1.3,.2)(2.9,.2)

\Rput[ul](2.9,.2){${\rm U}$}

\psline{->}(3.5,.2)(5.1,.2)

\psline{<-}(3.5,-1.1)(5.5,-1.1)

\Rput[ul](5.5,-1.1){${\rm U}$}

\psline{->}(6,-1.1)(8.1,-1.1)

\Rput[ul](4.4,-1.8)
{$\tau.((y_{_\beta})_{_\beta},(x_{_a})_{_a})\in U$}

\qdisk(4.25,0.5){2pt}

\Rput[ul](3.8,.8){$t$}

\Rput[ul](2.1,-.5)
{$\tau.(t,(x_{_a})_{_a})\in U$}

\Rput[ul](9.7,1){$(y_{_\beta})_{_{\beta\in B}}$}

\psline[linestyle=dotted]{<-}(6.5,.5)(11.3,.5)

\psline{<-}(6.5,.2)(8.1,.2)

\Rput[ul](8.0,.2){${\rm U}$}

\psline{->}(8.5,.2)(10.1,.2)

\qdisk(7.2,0.5){2pt}

\Rput[ul](7.1,.8){$t$}

\Rput[ul](7.3,-.5)
{$\tau.(t,(y_{_\beta})_{_\beta})\in U$}

\endpspicture
\par\bigskip\bigskip\bigskip\bigskip\bigskip\bigskip\bigskip\bigskip\bigskip
\begin{center}
Figure 3
\end{center}

\begin{remark}\label{a2}{\rm It is not obligatory for a convergent net to get subnets
which are $\tau_{_p}$-nets or $\tau_{_q}$-nets.}
\end{remark}
\begin{definition}\label{a3}{\rm Let ${\mathcal {(A}},{\mathcal B})$ be an ordered couple whose members are non-empty
families of $\tau_{_p}$- and $\tau_{_q}$-nets respectively.
We say that ${\mathcal {(A}},{\mathcal B})$ is a $\tau$-{\it {cut {\rm ({\it of nets})}}}
if the following conditions are fulfilled:

(1) for every $U\in {\mathcal{U}}$, every $(x^i_a)_{a\in A_i}\in \mathcal A$ and
every $(y\!_{_{_{\beta}}}\!\!^j)_{_{_{\beta}\in B_j}}\in {\mathcal B}$
there holds
\begin{center}
$\tau.((y\!_{_{_{\beta}}}\!\!^{^j})_{_\beta}, (x_a^i)_{_a})\in U$.
\end{center}

(2) ${\mathcal B}$ contains all the $\tau_{_q}$-nets which are conets of all the $\tau_{_p}$-nets
of ${\mathcal A}$ and conversely: ${\mathcal A}$ contains
all the $\tau_{_p}$-nets whose conets are all the $\tau_{_q}$-nets elements of
${\mathcal B}$.

We call the member $\mathcal A$ (resp. ${\mathcal B})$ {\it first}
(resp.{\it second) class of the $\tau$-cut} and the elements of ${\mathcal{A}}$ and ${\mathcal{B}}$,  {\it elements of the
$\tau$-cut}.}
\end{definition}
Throughout the paper, for simplicity of the proofs,
we call the elements of ${\mathcal{B}}$ $\tau_{_q}$-{\it conets} of ${\mathcal{A}}$.
Moreover, by saying that a $\tau_{_p}$-net (resp. $\tau_{_q}$-conet) member of the first (resp. second) class of a $\tau$-cut converges to a point $x\in X$,
we mean that it converges with respect to $\tau(\mathcal{U})$ (resp. $\tau(\mathcal{U}^{-1})$)
(see fig. 4 and 6).

\begin{center}
$(\mathcal{A},\mathcal{B})$\ $\tau$-cut
\end{center}

\pspicture(0,0)(0,0)

\Rput[ul](2.2,-.1){$\mathcal{A}$}

\Rput[ul](8,-.1){$\mathcal{B}$}
\psline{-}(0.5,-.0)(5.1,-.7)

\psline{-}(0.5,-1)(5.1,-.7)

\qdisk(1,-.6){1pt}

\qdisk(1,-1.3){1pt}

\qdisk(1,-2){1pt}

\qdisk(1.7,-.25){1pt}

\qdisk(1.9,-.3){1pt}

\qdisk(2.1,-.5){1pt}

\qdisk(2.3,-.7){1pt}

\qdisk(2.5,-.7){1pt}

\qdisk(2.7,-.5){1pt}

\qdisk(2.9,-.5){1pt}

\qdisk(3.1,-.6){1pt}

\qdisk(3.3,-.7){1pt}

\qdisk(3.5,-.6){1pt}

\qdisk(3.7,-.7){1pt}

\psline[linestyle=dotted]{-}(3.9,-.7)(5.1,-.7)

\psline[linestyle=dotted]{-}(3.9,-1.34)(5.1,-1.3)

\psline[linestyle=dotted]{-}(3.9,-1.87)(5.1,-1.8)

\qdisk(1.7,-1.25){1pt}

\qdisk(1.9,-1.3){1pt}

\qdisk(2.1,-1.2){1pt}

\qdisk(2.3,-1.1){1pt}

\qdisk(2.5,-1.2){1pt}

\qdisk(2.7,-1.3){1pt}

\qdisk(2.9,-1.4){1pt}

\qdisk(3.1,-1.35){1pt}

\qdisk(3.3,-1.37){1pt}

\qdisk(3.5,-1.31){1pt}

\qdisk(3.7,-1.34){1pt}

\qdisk(1.6,-1.9){1pt}

\qdisk(1.9,-2.1){1pt}

\qdisk(2.1,-1.8){1pt}

\qdisk(2.3,-1.9){1pt}

\qdisk(2.5,-2){1pt}

\qdisk(2.7,-1.78){1pt}

\qdisk(2.9,-1.8){1pt}

\qdisk(3.1,-1.85){1pt}

\qdisk(3.3,-1.89){1pt}

\qdisk(3.5,-1.82){1pt}

\qdisk(3.7,-1.87){1pt}

\psline[linestyle=dotted]{-}(3.9,-1.34)(5.1,-1.3)

\psline[linestyle=dotted]{-}(3.9,-1.87)(5.1,-1.8)

\psline{-}(0.5,-.7)(5.1,-1.3)

\psline{-}(0.5,-1.7)(5.1,-1.3)

\Rput[ul](.9,-.5){$x^i_{_{a_{_{_U}}}}$}

\Rput[ul](.9,-1.29){$x^j_{_{a_{_{_U}}}}$}

\Rput[ul](.9,-1.95){$x^k_{_{a_{_{_U}}}}$}

\psline{-}(0.5,-1.6)(5.1,-1.8)

\psline{-}(0.5,-2.4)(5.1,-1.8)

\psline{-}(10.7,-.0)(5.7,-.7)

\psline{-}(10.7,-1)(5.7,-.7)

\psline{-}(10.7,-.8)(5.7,-1.2)

\psline{-}(10.7,-1.8)(5.7,-1.2)

\psline{<-}(1.1,-2.6)(5.3,-2.6)

\Rput[ul](5.25,-2.6){${\rm U}$}

\psline{->}(5.8,-2.6)(10,-2.6)

\Rput[ul](9.7,-1.35){$y^j_{_{\beta_{_{_U}}}}$}

\Rput[ul](9.7,-.5){$y^i_{_{\beta_{_{_U}}}}$}

\qdisk(9.8,-.6){1pt}

\qdisk(9.5,-.7){1pt}

\qdisk(9.2,-.7){1pt}

\qdisk(8.9,-.6){1pt}

\qdisk(8.7,-.8){1pt}

\qdisk(8.5,-.56){1pt}

\qdisk(8.3,-.42){1pt}

\qdisk(8.1,-.78){1pt}

\qdisk(7.9,-.69){1pt}

\qdisk(7.7,-.55){1pt}

\qdisk(7.5,-.60){1pt}

\qdisk(7.3,-.58){1pt}

\qdisk(9.8,-1.4){1pt}

\qdisk(9.5,-1.5){1pt}

\qdisk(9.2,-1.45){1pt}
\qdisk(8.9,-1.5){1pt}
\qdisk(8.7,-1.2){1pt}
\qdisk(8.5,-1.36){1pt}
\qdisk(8.3,-1.42){1pt}

\qdisk(8.1,-1.38){1pt}

\qdisk(7.9,-1.28){1pt}

\qdisk(7.7,-1.35){1pt}

\qdisk(7.5,-1.28){1pt}

\qdisk(7.3,-1.26){1pt}

\psline[linestyle=dotted]{-}(7.3,-0.58)(5.7,-.7)

\psline[linestyle=dotted]{-}(7.3,-1.26)(5.7,-1.2)

\endpspicture
\par\bigskip\bigskip\bigskip\bigskip\bigskip\bigskip\bigskip\bigskip\bigskip
\begin{center}
Figure 4
\end{center}
\par\bigskip\smallskip\par

If all the elements of $\mathcal{A}$ and $\mathcal{B}$ converge to a point $x\in X$,
then $\{x\}$ belongs to both of the classes
$\mathcal{A}$ and $\mathcal{B}$. In this case, we call $x$ {\it end of the two classes} or we say
that {\it both of the classes have} $x$ {\it as an end point},
or we simply say that $x$ {\it is an end point of the $\tau$-cut
$(\mathcal{A},\mathcal{B})$}.
If there is an end point we say that
$(\mathcal{A},\mathcal{B})$ {\it converges} to $x$.
In the case of a uniform space the two classes coincide.

If the classes of $(\mathcal{A},\mathcal{B})$ have not an end point we say that $(\mathcal{A},\mathcal{B})$
is a $\tau$-{\it gap}. The set of all $\tau$-gaps of $X$ is symbolized by
$\Lambda(X)$.

\begin{notation}\label{a4}{\rm  We symbolize the set of all $\tau$-cuts of $X$ by $\overline{X}$.
For every $\xi\in \overline{X}$, we put $\xi=({\mathcal{A}}
_{\xi},{{\mathcal{B}}}_{\xi})$, ${{\mathcal{A}}}_{\xi}$,
${{\mathcal{B}}}_{\xi}$ being the two classes of the $\tau$-cut
${\xi}$. We will preserve this notation and terminology throughout the
paper (see fig. 5).}
\end{notation}

\par\bigskip\bigskip\bigskip\bigskip\par

\pspicture(0.5,0.5)(1.2,1.2)

\psline{->}(2.3,.2)(5,.2)

\psline{->}(2.3,.5)(5,.5)

\psline{->}(2.3,-.1)(5,-.1)

\psline{->}(5.5,-.4)(5.5,.9)

\psline{<-}(5.9,.5)(8.8,.5)

\psline{<-}(5.9,.2)(8.8,.2)

\Rput[ul](5.25,1.2)
{${\xi}$}

\Rput[ul](3.3,.8){${\mathcal{A}}_{_{\xi}}$}

\Rput[ul](6.9,.8){${\mathcal{B}}_{_{\xi}}$}

\endpspicture
\par\bigskip\bigskip\bigskip\bigskip
\par
\ \ \ \ \ \ \ \ \ \ \ \ \ \ \ \ \ \ \ \ \ \ \ \ \ \ \ \ \ \ \ \ \ \ \ \ \ \
Figure 5
\par
\par\bigskip\smallskip\par

\smallskip
\par
In a $T_{_0}$ space $X$, for any two elements $x$ and $y$, there is
an entourage, say $U$, such that the one of the points, say $y$,
does not belong to $U(x)$. It means that we may uniquely correspond
to every point $x\in X$ a $\tau$-cut $\phi(x)=({{\mathcal{A}}}_{_{\phi(x)}},{{\mathcal{B}}}_{_{\phi(x)}})$, where
all the points of
${{\mathcal{A}}}_{_{\phi(x)}}$ and
${{\mathcal{B}}}_{_{\phi(x)}}$
converge to $x$. We recall that the point $x$
itself, which is the \textquotedblleft end\textquotedblright of all
these $\tau_{_p}$-nets and $\tau_{_q}$-conets, belongs to both of the classes (see fig. 6).

We call the above map $\phi: X\longrightarrow \overline{X}$
\textquotedblleft{\it the canonical embedding of $X$ into}
$\overline{X}$\textquotedblright and we preserve the notation and
the meaning of that $\phi$ throughout all the paper.

So, from one point of view we may, roughly speaking, consider that
$\phi(x)$ and $x\in X$ coincide, giving a reason why we write
$X\subseteq \overline{X}$. By definition, $\overline{X}=\phi(X)\cup \Lambda(X)$.

\par\noindent

\bigskip\bigskip\par

\pspicture(0,0)(0,0)

\Rput[ul](2.3,-.1){${\mathcal{A}}_{_{\phi(x)}}$}

\Rput[ul](8.3,-.1){${\mathcal{B}}_{_{\phi(x)}}$}

\Rput[ul](5.17,-1.55){$x$}

\qdisk(1.7,-.25){1pt}

\qdisk(1.9,-.3){1pt}

\qdisk(2.1,-.5){1pt}

\qdisk(2.3,-.7){1pt}

\qdisk(2.5,-.7){1pt}

\qdisk(2.7,-.5){1pt}

\qdisk(2.9,-.6){1pt}

\qdisk(3.1,-.9){1pt}

\qdisk(3.3,-.9){1pt}

\qdisk(3.5,-.85){1pt}

\psline[linestyle=dotted]{-}(3.5,-.85)(5.1,-1.3)

\psline[linestyle=dotted]{-}(3.9,-1.34)(5.1,-1.3)

\qdisk(1.9,-1.3){1pt}

\qdisk(2.1,-1.2){1pt}

\qdisk(2.3,-1.1){1pt}

\qdisk(2.5,-1.2){1pt}

\qdisk(2.7,-1.3){1pt}

\qdisk(2.9,-1.4){1pt}

\qdisk(3.1,-1.35){1pt}

\qdisk(3.3,-1.37){1pt}

\qdisk(3.5,-1.31){1pt}

\qdisk(3.7,-1.34){1pt}

\qdisk(1.9,-2.1){1pt}

\qdisk(2.1,-1.8){1pt}

\qdisk(2.3,-1.9){1pt}

\qdisk(2.5,-2){1pt}

\qdisk(2.7,-1.78){1pt}

\qdisk(2.9,-1.8){1pt}

\qdisk(3.1,-1.85){1pt}

\qdisk(3.3,-1.7){1pt}

\qdisk(3.5,-1.6){1pt}

\qdisk(3.7,-1.65){1pt}

\psline[linestyle=dotted]{-}(3.7,-1.65)(5.1,-1.3)

\psline{->}(5.4,-1.3)(5.4,-.2)

\Rput[ul](5.5,-.3){$\phi(x)$}

\psline{-}(0.5,0.3)(5.4,-1.3)

\psline{-}(0.5,-.7)(5.4,-1.3)

\psline{-}(0.5,-.85)(5.4,-1.3)

\psline{-}(0.5,-1.7)(5.4,-1.3)

\psline{-}(0.5,-1.9)(5.4,-1.3)

\psline{-}(0.5,-2.7)(5.4,-1.3)

\psline{-}(10.7,-.0)(5.4,-1.3)

\psline{-}(10.7,-1)(5.4,-1.3)

\psline{-}(10.7,-.8)(5.4,-1.3)

\psline{-}(10.7,-1.8)(5.4,-1.3)

\Rput[ul](-.1,-.3){$(x^i_{_a})_{_{a\in A_{_i}}}$}

\Rput[ul](-.1,-1.3){$(x^j_{_a})_{_{a\in A_{_j}}}$}

\Rput[ul](-.1,-2.2){$(x^k_{_a})_{_{a\in A_{_k}}}$}

\Rput[ul](9.8,-1.35){$(y^j_{_{\beta}})_{_{\beta\in B_{_j}}}$}

\Rput[ul](9.8,-.5){$(y^i_{_{\beta}})_{_{\beta\in B_{_i}}}$}

\qdisk(9.8,-.6){1pt}

\qdisk(9.5,-.7){1pt}

\qdisk(9.2,-.7){1pt}

\qdisk(8.9,-.6){1pt}

\qdisk(8.7,-.8){1pt}

\qdisk(8.5,-.56){1pt}

\qdisk(8.3,-.8){1pt}

\qdisk(8.1,-.78){1pt}

\qdisk(7.9,-.91){1pt}

\qdisk(7.7,-.95){1pt}

\qdisk(7.5,-.9){1pt}

\qdisk(7.3,-1){1pt}

\psline[linestyle=dotted]{-}(7.3,-1)(5.7,-1.3)

\qdisk(9.8,-1.4){1pt}

\qdisk(9.5,-1.5){1pt}

\qdisk(9.2,-1.45){1pt}
\qdisk(8.9,-1.5){1pt}
\qdisk(8.7,-1.2){1pt}
\qdisk(8.5,-1.46){1pt}
\qdisk(8.3,-1.42){1pt}

\qdisk(8.1,-1.38){1pt}

\qdisk(7.9,-1.28){1pt}

\qdisk(7.7,-1.35){1pt}

\qdisk(7.5,-1.28){1pt}

\qdisk(7.3,-1.32){1pt}

\psline[linestyle=dotted]{-}(7.3,-1.32)(5.7,-1.3)

\psline{->}(1.8,-3)(5.3,-3)

\psline{<-}(5.8,-3)(9,-3)

\Rput[ul](3.5,-2.65){$\tau(\mathcal{U})$}

\Rput[ul](6.4,-2.65){$\tau(\mathcal{U}^{-1})$}

\endpspicture

\par\bigskip\bigskip\bigskip\bigskip\bigskip\bigskip\bigskip\bigskip\bigskip\bigskip\bigskip
\begin{center}
Figure 6
\end{center}
\par\bigskip\bigskip

\begin{remark}\label{a5}{\rm (i) There is a  correspondence between a point of a $T_0$-space and a
$\tau$-cut
(ii) If the $\tau_{_p}$-nets of the first class of a $\tau$-cut converge to a point $x$
and the $\tau_{_q}$-conets of the second class do not converge to $x$, then the $\tau$-cut is not the $\phi(x)$.
In this case the $\tau$-cut is a $\tau$-gap.

Consider the following example: in a $\tau$-cut $\xi=(\mathcal{A}_{_\xi},\mathcal{B}_{_\xi})$, the class $\mathcal{A}_{_\xi}$
consists of a $\tau_{_p}$-net $A$ and the second class $\mathcal{B}_{_\xi}$ consists of a
$\tau_{_q}$-conet $B$ and of a singleton $\{\beta\}$ as another $\tau_{_q}$-conet.
Then, the singleton $\{\beta\}$ itself constitutes another $\tau$-cut, the $\phi({\beta})$, different of the $\tau$-cut
$(\mathcal{A}_{_\xi},\mathcal{B}_{_\xi})$.}
\end{remark}

\begin{proposition}\label{a6}{\rm For every $\tau$-cut $({\mathcal{A}},{\mathcal{B}})$, every $U\in
{\mathcal{U}}$ and every $(x_a\!\!\!^i)_{a\in A_i}\in {\mathcal{A}}$,
$(y_{_{\beta}}\!\!^{^j})_{_{\beta\in B_j}}\in {\mathcal{B}}$, there are
fixed indices $a\!_{_{_U}}\!\!\!^i\in A_i$,
${\beta}\!_{_U}\!\!^{^j}\in B_j$ such that
$\tau.(y_{_{\beta}}\!\!^{^j},x_a\!\!\!^i)\in U$ for every $a\geq
a\!_{_{_U}}\!\!\!^i$ and $\beta\geq {\beta}\!_{_U}\!\!^{^j}$ (see fig. 7).}
\end{proposition}
\begin{proof}Let $V,U\in {\mathcal{U}}$ such that $V\circ V\circ
V\subseteq U$. Let also $(x_a\!\!\!^i)_{a\in A_i}\in {\mathcal{A}}$,
$(y_{_{\beta}}\!\!^{^j})_{_{\beta\in B_j}}\in {\mathcal{B}}$ and
$a\!_{_{_V}}\!\!\!\!^i$, ${\beta}\!_{_V}\!\!\!^{^j}$ two indices
such that
 $(x_a^i,x_{a^{\prime}}^i)\in V$,
$(y_{_\beta},y_{_{\beta^{\prime}}}^j)\in V$ for each
$a,a^{\prime}\geq a\!_{_{_V}}\!\!\!\!^i$ and
$\beta,\beta^{\prime}\geq {\beta}\!_{_V}\!\!\!^{^j}$. Then from
$\tau.((y_{_{\beta}}\!\!^{^j})_{_{\beta}},(x_a\!\!\!^i)_a)\in V$ we
conclude that $(y_{_{\beta}}\!\!^{^j},x_a\!\!\!^i)\in U$ for each
$i\in I,j\in J, a\geq a\!_{_{_V}}\!\!\!\!^i$ and $\beta\geq
{\beta}\!_{_V}\!\!\!^{^j}$.
\end{proof}

\pspicture(0,0)(0,0)

\psline{-}(0.3,-.3)(5.1,-1.3)

\psline{-}(0.3,-1.8)(5.1,-1.3)

\Rput[ul](.9,-1.29){$x^i_{_{a_{_{_U}}}}$}

\qdisk(1,-1.4){1pt}

\qdisk(2.2,-1.5){1pt}

\Rput[ul](1.9,-1.2){$x^i_{{a^{^\prime}}}$}

\qdisk(1,-1.4){1pt}

\Rput[ul](9.9,-1.25){$y^j_{_{\beta_{_U}}}$}

\qdisk(9.93,-1.42){1pt}

\qdisk(8.7,-1.45){1pt}

\Rput[ul](8.5,-1.1){$y^i_{_{\beta^{\prime}}}$}

\qdisk(8.1,-1.45){1pt}

\Rput[ul](7.7,-1.1){$y^i_{_\beta}$}

\qdisk(2.9,-1.4){1pt}

\Rput[ul](2.8,-1.2){$x^i_{{a}}$}

\psline{-}(10.7,-.3)(5.7,-1.3)

\psline{-}(10.7,-1.8)(5.7,-1.3)

\psline{<-}(1.1,-2.4)(5.3,-2.4)

\Rput[ul](5.25,-2.4){${\rm U}$}

\psline{->}(5.8,-2.4)(10,-2.4)

\endpspicture
\par\bigskip\bigskip\bigskip\bigskip\bigskip\bigskip\bigskip\bigskip\bigskip\bigskip
\begin{center}
Figure 7
\end{center}
\par\bigskip\smallskip\par

\par\noindent
\begin{remark}\label{a7}{\rm The existence in the Proposition \ref{a6} fixed extreme
points for a given entourage $U$ and for all pairs of the
$\tau_{_p}$-nets and their $\tau_{_q}$-conets is a crucial property of the so
called {\it quiet spaces}. It is a result coming from the $\mathcal{Q}$-property
(the property of the quiet spaces)
and the proposition 12 of \cite{doi2} and it is the
basic reason of why these spaces may be completed in such a simple
way by the D-completion.}
\end{remark}

The following proposition is evident.

\begin{proposition}\label{a800}{\rm A $\tau$-cut is not a $\tau$-gap if and only if all the members of the two classes
converge to the same point.}
\end{proposition}

\begin{definition}\label{a8}{\rm We call $\tau$-{\it Cauchy net}
every $\tau_{_p}$-net or $\tau_{_q}$-conet of a $\tau$-cut.}
\end{definition}

\begin{definition}\label{a9}{\rm A quasi-uniform space is called $\tau$-{\it complete} if
all the $\tau$-Cauchy nets of a $\tau$-cut converge to the same point.}
\end{definition}

After the Proposition \ref{a800} an equivalent definition of the $\tau$-completeness holds:

\begin{definition}\label{a900}{\rm A quasi-uniform space is called $\tau$-{\it complete} if
its $\tau$-cuts are not $\tau$-gaps.}
\end{definition}

\par
\begin{definition}\label{a10}{\rm Let $(X,{\mathcal{U}})$ be a quasi-uniform space and let
$(x_{_a})_{_{a\in A}}$ and $(x_{_\beta})_{_{\beta\in B}}$ be two nets in $X$.
Given a $W\in {\mathcal{U}}$ we say that
$(x_{_a})_{_{a\in A}}$ and $(x_{_\beta})_{_{\beta\in B}}$
are
{\it left cofinal for $W$} if and only if
there are $\beta_{_W}\in B, a_{_W}\in A$ satisfying the
following property: for every $a\geq a_{_W}$ there exists
$\beta_{_a}>\beta_{_W}$ such that for every $\beta>\beta_{_a}$ there holds
$(x_{_\beta},x_{_a})\in W$ and for
every $\beta\geq \beta_{_W}$ there exists $a_{_\beta}>a_{_W}$ such that for
every $a>a_{_\beta}$ there holds $(x_{_a},x_{_\beta})\in W$.

We say that $(x_{_a})_{_{n\in A}}$ and $(x_{_\beta})_{_{\beta\in
B}}$ are {\it left cofinal}, if there exists $W_{_0}\in \mathcal{U}$ such that for each $W\in \mathcal{U}$ and $W\subseteq W_{_0}$,
$(x_{_a})_{_{a\in A}}$ and $(x_{_\beta})_{_{\beta\in B}}$
are
left cofinal
for $W$ (see fig 8).

Dually we define the {\it right cofinality} of
$(x_{_a})_{_{a\in A}}$ and $(x_{_\beta})_{_{\beta\in B}}$.}
\end{definition}

\begin{proposition}\label{a11}{\rm Let $(x_a)_{_{a\in A}}$ be a $\tau_{_p}$-net (resp.$(y_{_\beta})_{_{\beta\in B}}$ is
a $\tau_{_q}$-net) in a quasi-uniform space $(X,{\mathcal{U}})$ and
$(x_{_{a_i}})_{_{i\in I}}$ a $\tau_{_p}$-subnet (resp.
$(y_{_{\beta_j}})_{_{j\in J}}$ a $\tau_{_q}$-subnet) of it. Then
$(x_a)_{_{a\in A}}$ and $(x_{_{a_i}})_{_{i\in I}}$ (resp.
$(y_{_\beta})_{_{\beta\in B}}$ and $(y_{_{\beta_j}})_{_{j\in J}}$)
are left (resp. right) cofinal.}
\end{proposition}

\par\bigskip\par

\pspicture(0,0)(-1,-1)

\Rput[ul](-1.2,-.6){$(x_{_a})_{_{a\in A}}$}

\Rput[ul](.7,-.6){$x_{_{a_{_W}}}$}

\Rput[ul](2,-.6){$x_{_a}$}

\Rput[ul](-1.2,-1.5){$(x_{_\beta})_{_{\beta\in B}}$}

\Rput[ul](.7,-1.5){$x_{_{\beta_{_W}}}$}

\Rput[ul](2.3,-1.5){$x_{_{\beta_{_a}}}$}

\psline[linestyle=dotted]{-}(-1,-.9)(1.2,-.9)
\qdisk(1,-.9){1pt}

\qdisk(2.3,-.9){2pt}

\qdisk(1,-1.8){1pt}

\qdisk(2.55,-1.8){1pt}

\psline{<-}(2.5,-2)(3,-2)

\Rput[ul](2.9,-2.1){$x_{_\beta}$}

\psline{->}(3.4,-2)(4,-2)

\psline[linestyle=dotted]{->}(1.2,-.9)(4,-.9)

\psline[linestyle=dotted]{-}(-1,-1.8)(1.2,-1.8)

\psline[linestyle=dotted]{->}(1.2,-1.8)(4,-1.8)

\psline{<-}(1.02,-2.4)(2.2,-2.4)

\Rput[ul](2.15,-2.4){${\rm W}$}

\psline{->}(2.8,-2.4)(4,-2.4)

\psline[linestyle=dotted]{-}(-1,-.9)(1.2,-.9)
\qdisk(1,-.9){1pt}

\psline{<-}(8.2,-2)(8.9,-2)

\Rput[ul](8.9,-2.1){$x_{_a}$}

\psline{->}(9.5,-2)(10.1,-2)

\Rput[ul](4.9,-1.5){$(x_{_a})_{_{a\in A}}$}

\Rput[ul](6.8,-1.5){$x_{_{a_{_W}}}$}

\qdisk(7.1,-1.8){1pt}

\Rput[ul](8.1,-1.5){$x_{_{a_{_\beta}}}$}

\qdisk(8.25,-1.8){1pt}

\qdisk(8.1,-.9){2pt}

\Rput[ul](7.8,-.6){$x_{_\beta}$}

\qdisk(7,-.9){1pt}
\Rput[ul](6.5,-.6){$x_{_{\beta_{_W}}}$}

\Rput[ul](4.8,-.6){$(x_{_\beta})_{_{\beta\in B}}$}

\psline[linestyle=dotted]{-}(5,-.9)(6.2,-.9)

\psline[linestyle=dotted]{->}(6.2,-.9)(10.1,-.9)

\psline[linestyle=dotted]{-}(5,-1.8)(7.2,-1.8)

\psline[linestyle=dotted]{->}(7.2,-1.8)(10.1,-1.8)

\psline{<-}(7.02,-2.4)(8.2,-2.4)

\Rput[ul](8.15,-2.4){${\rm W}$}

\psline{->}(8.8,-2.4)(10.1,-2.4)

\Rput[ul](.4,-3.1){$(x_{_\beta},x_{_a})\in W$}

\Rput[ul](6.6,-3.1){$(x_{_a},x_{_\beta})\in W$}

\endpspicture
\par\bigskip\bigskip\bigskip\bigskip\bigskip\bigskip\bigskip
\begin{center}
Figure 8
\end{center}
\par\

\begin{proposition}\label{a12}{\rm In a quasi-uniform space two left cofinal $\tau_{_p}$-nets (resp. cofinal $\tau_{_q}$-conets)
have the same $\tau_{_q}$-conets (resp. $\tau_{_p}$-nets).}
\end{proposition}
\begin{proof} Let $(x_{_a})_{_{a\in A}}$, $(t_{_k})_{_{k\in K}}$ be two left cofinal $\tau_{_p}$-nets and
$(y_{_\beta})_{_{\beta\in B}}$ is a $\tau_{_q}$-conet of
$(x_{_a})_{_{a\in A}}$. Let also $U\in {\mathcal{U}}$ and a $W\in {\mathcal{U}}$ with $W\circ W\subseteq U$. Then there exist $a_{_W}\in A$ and
$\beta_{_W} \in B$ such that $(y_{_\beta},x_a)\in W$ for $a\geq
a_{_W}$ and $\beta\geq \beta_{_W}$. On the other hand, because of
left cofinality of $(x_{_a})_{_{a\in A}}$ and $(t_{_k})_{_{k\in
K}}$, there are $a^{\prime}_{_W}\in A$ and $k_{_W}\in K$ with the
property: for each $k>k_{_W}$ there exists $a_k>a^{\prime}_{_W}$
such that for each $a>a_k$ we have $(x_a,t_k)\in W$. Hence,
$(y_{_\beta}, t_{_k})\in W\circ W\subseteq U$,
 whenever
$\beta\geq \beta_{_W}$ and $k\geq k_{_W}$.

Likewise we prove the result for the right cofinal $\tau_{_q}$-nets.
\end{proof}

\par
\begin{proposition}\label{a13}{\rm In a quasi-uniform space $(X,{\mathcal{U}})$ two left cofinal $\tau_{_p}$-nets (resp. right cofinal $\tau_{_q}$-conets)
have the same limit points.}
\end{proposition}

\par

\bigskip
\par\noindent
\begin{remark}\label{a14}
\par\noindent
\begin{enumerate}
\item
{\rm Without loss of generality, we may suppose that for $U\subseteq
V$, it is $a_{_U}\!\!\!\!^i\geq a_{_V}\!\!\!\!^i$ and
${\beta}_{_U}\!\!\!^{^j}\geq {\beta}_{_V}\!\!\!\!^{^i}$\ for the
corresponded extreme points of $(x_a\!\!\!^i)_{a\in A_i}$ and
$(y_{_{\beta}}\!\!^{^j})_{_{\beta\in B_j}}$ respectively.
\item
Given one or more $\tau_{_p}$-nets of a $\tau$-cut we can construct the
$\tau$-cut taking all the $\tau_{_q}$-conets of the given and after that,
all the $\tau_{_p}$-nets of these $\tau_{_q}$-conets. The procedure may be
reversed considering firstly the $\tau_{_q}$-conets. In the present text
we preserve the former procedure.}
\end{enumerate}
\end{remark}

\section{The $\tau$-completion procedure}

\par
The pair $(X,{\mathcal{U}})$ always presents a $T_{_0}$ quasi-uniform
space, $\overline{X}$ is the set of all $\tau$-cuts in $X$ and
$\phi$ the set-theoretical embedding of $X$ in $\overline{X}$
described above in the notation \ref{a4} and the remark \ref{a5}, the {\it canonical embedding of}
$X$ {\it into} $\overline{X}$ as we have called it. We shall define
a quasi-uniformity on $\overline{X}$.

\begin{definition}\label{a15}{\rm For any $\tau$-cut $\xi^{^\prime}=({{\mathcal{A}}} _{\xi
^{^\prime}},{{\mathcal{B}}}_{\xi^{^\prime}})\in \overline{X}$ and  any
$W\in {\mathcal{U}}$ we define $\overline{W}(\xi^{^\prime})$ as the set
of all $\xi^{^{\prime\prime}}=({{\mathcal{A}}}_{\xi ^{^{\prime\prime}}},
{{\mathcal{B}}}_{\xi^{^{\prime\prime}}})\in \overline{X}$ which fulfil the
following}:
\begin{enumerate}
\item
{\rm exclusively} ${\mathcal{A}} _{\xi ^{^{\prime\prime}}}\subseteq{\mathcal{A}} _{\xi
^{^{\prime}}}$ {\rm or}
\item
{\rm there is a $\tau_{_p}$-net $(x_{_\beta})_{_{\beta\in B}}\in  A _{\xi
^{^\prime}}$ such that
$\tau.((x_{_\beta})_{_\beta},(x_{_a})_{_a})\in W$ for every
$(x_{_a})_{_{a\in A}}\in {\mathcal{A}} _{\xi ^{^{\prime\prime}}}$ (see fig. 9).}
\end{enumerate}
\end{definition}

\par
We put $\overline{\mathcal{U}}=\{ \overline{W}: W \in {\mathcal{U}} \}$.

\par\smallskip\par

Let ${\mathcal{U}}_{_0}$ be again the base of the quasi-uniformity the
referred of the beginning of the paragraph 2. In the sequel, to every
entourage in $(X,{\mathcal{U}})$, say $U\in {\mathcal{U}}$, the Definition
\ref{a15}
corresponds a subset of $\overline{X}\times \overline{X}$, that is
an \textquotedblleft entourage\textquotedblright\ in
$(\overline{X},\overline{{\mathcal{U}}})$; we symbolize this subset by
$\overline{U}$, that is by the same letter as in $X$ putting a bar.

\par\bigskip\bigskip\bigskip\bigskip\par

\pspicture(0,0)(0,0)

\Rput[ul](3,-.6){${{\mathcal{A}}}_{\xi ^{^{\prime}}}$}

\Rput[ul](8,-.6){${{\mathcal{B}}}_{\xi ^{^{\prime}}}$}

\Rput[ul](.8,-.7){${{\mathcal{A}}}_{\xi ^{^{\prime\prime}}}$}

\Rput[ul](5,-.6){${{\mathcal{B}}}_{\xi ^{^{\prime\prime}}}$}

\psline{-}(0.3,-.1)(5.1,-1.3)

\psline{-}(0.3,-2)(5.1,-1.3)

\psline{-}(0.3,-.8)(3.1,-1.3)

\psline{-}(0.3,-1.4)(3.1,-1.3)

\psline{-}(10.7,0.2)(3.5,-1.3)

\psline{-}(10.7,-2.3)(3.5,-1.3)

\psline{->}(3.35,-1.7)(3.35,-1)

\psline{->}(5.5,-1.7)(5.5,-1)

\psline{-}(10.7,-.5)(5.8,-1.3)

\psline{-}(10.7,-1.5)(5.8,-1.3)

\Rput[ul](2.8,-1.05)
{${\xi^{\prime\prime}}$}

\Rput[ul](5.55,-1.05)
{${\xi^{\prime}}$}

\Rput[ul](4.5,-3){${{\mathcal{A}}}_{\xi ^{^{\prime\prime}}}\subseteq {{\mathcal{A}}}_{\xi ^{^{\prime}}}$}

\endpspicture
\par
\bigskip\bigskip\bigskip\bigskip\bigskip\bigskip\bigskip\bigskip\bigskip\bigskip
\par\noindent
or
\par\bigskip\bigskip\par\noindent

\par\noindent

\pspicture(0,0)(0,0)

\Rput[ul](2.5,-.1){${\mathcal{A}}_{_{\xi^{^{\prime\prime}}}}$}

\psline{-}(0.5,-.0)(5.1,-.7)

\psline{-}(0.5,-1)(5.1,-.7)

\qdisk(1,-.6){1pt}

\qdisk(1,-1.3){1pt}

\qdisk(1,-2){1pt}

\qdisk(1.7,-.25){1pt}

\qdisk(1.9,-.3){1pt}

\qdisk(2.1,-.5){1pt}

\qdisk(2.3,-.7){1pt}

\qdisk(2.5,-.7){1pt}

\qdisk(2.7,-.5){1pt}

\qdisk(2.9,-.5){1pt}

\qdisk(3.1,-.6){1pt}

\qdisk(3.3,-.7){1pt}

\qdisk(3.5,-.6){1pt}

\qdisk(3.7,-.7){1pt}

\psline[linestyle=dotted]{-}(3.9,-.7)(5.1,-.7)

\psline[linestyle=dotted]{-}(3.9,-1.34)(5.1,-1.3)

\psline[linestyle=dotted]{-}(3.9,-1.87)(5.1,-1.8)

\qdisk(1.7,-1.25){1pt}

\qdisk(1.9,-1.3){1pt}

\qdisk(2.1,-1.2){1pt}

\qdisk(2.3,-1.1){1pt}

\qdisk(2.5,-1.2){1pt}

\qdisk(2.7,-1.3){1pt}

\qdisk(2.9,-1.4){1pt}

\qdisk(3.1,-1.35){1pt}

\qdisk(3.3,-1.37){1pt}

\qdisk(3.5,-1.31){1pt}

\qdisk(3.7,-1.34){1pt}

\qdisk(1.6,-1.9){1pt}

\qdisk(1.9,-2.1){1pt}

\qdisk(2.1,-1.8){1pt}

\qdisk(2.3,-1.9){1pt}

\qdisk(2.5,-2){1pt}

\qdisk(2.7,-1.78){1pt}

\qdisk(2.9,-1.8){1pt}

\qdisk(3.1,-1.85){1pt}

\qdisk(3.3,-1.89){1pt}

\qdisk(3.5,-1.82){1pt}

\qdisk(3.7,-1.87){1pt}

\psline[linestyle=dotted]{-}(3.9,-1.34)(5.1,-1.3)

\psline[linestyle=dotted]{-}(3.9,-1.87)(5.1,-1.8)

\psline{-}(0.5,-.7)(5.1,-1.3)

\psline{-}(0.5,-1.7)(5.1,-1.3)

\Rput[ul](.9,-.5){$x^i_{_{a}}$}

\Rput[ul](.9,-1.29){$x^j_{_{a}}$}

\Rput[ul](.9,-1.95){$x^k_{_{a}}$}

\psline{-}(0.5,-1.6)(5.1,-1.8)

\psline{-}(0.5,-2.4)(5.1,-1.8)

\psline{<-}(1.1,-2.6)(5.3,-2.6)

\Rput[ul](5.25,-2.6){${\rm W}$}

\psline{->}(5.8,-2.6)(10,-2.6)

\Rput[ul](6.7,-.8){$(x_{_\beta})_{_{\beta\in B}}\in {\mathcal{A}}_{_{\xi^{^{\prime}}}}$}

\qdisk(8.7,-1.2){1pt}
\qdisk(8.5,-1.36){1pt}
\qdisk(8.3,-1.42){1pt}

\qdisk(8.1,-1.38){1pt}

\qdisk(7.9,-1.28){1pt}

\qdisk(7.7,-1.35){1pt}

\qdisk(7.5,-1.28){1pt}

\qdisk(7.3,-1.26){1pt}

\qdisk(7.1,-1.4){1pt}

\qdisk(6.9,-1.5){1pt}

\qdisk(6.7,-1.45){1pt}
\qdisk(6.5,-1.5){1pt}

\psline[linestyle=dotted]{->}(8.7,-1.22)(9.9,-1.22)

\endpspicture

\par\bigskip\bigskip\bigskip\bigskip\bigskip\bigskip\bigskip\bigskip\bigskip
\begin{center}
Figure 9
\end{center}
\par\bigskip\smallskip\par

\par
The following theorem ensures the existence of a topology in $\overline{X}$.

\begin{theorem}\label{a16}{\rm The family $\overline{{\mathcal{U}}}=\{\overline{W}\vert W\in {\mathcal{U}}_{_0} \}$
is a base
for a quasi-uniformity on $\overline{X}$. We thus define a new
quasi-uniform space $(\overline{X},\overline{\mathcal{U}})$.}
\end{theorem}
\begin{proof} If  $U, V$ in  ${\mathcal{U}}_{_0}$ and $U\subseteq V$, then
$\overline{U}\subseteq \overline{V}$. In fact; if
$(\xi^{\prime},\xi^{\prime\prime})\in \overline{U}$ and
${{\mathcal{A}}}_{_{\xi^{\prime\prime}}}\subseteq {{\mathcal{A}}}_{_{\xi^{\prime}}}$ the
relation is evident. Otherwise, if
$(\xi^{\prime},\xi^{\prime\prime})\in \overline{U}$, there is a
$\tau_{_p}$-net $(x_{_\beta})_{_{\beta\in B}}$ of $\xi^{\prime}$ such
that $\tau.((x_{_\beta})_{_{\beta}},(x_{_a})_{_a})\in U$ for every
$(x_{_a})_{_{a\in A}}\in {{\mathcal{A}}}_{_{\xi^{\prime\prime}}}$, hence
$\tau.((x_{_\beta})_{_{\beta}},(x_{_a})_{_a})\in V$ and
 $(\xi^{\prime},\xi^{\prime\prime})\in \overline{V}$. From this result, we conclude
 that $\overline{U\cap V}\subseteq \overline{U}\cap \overline{V}$ and the
 considered family is a filter. We also have - by definition- that
 for every $\xi\in \overline{X}$ and every $\overline{W}$,
 $(\xi,\xi)\in \overline{W}$.
\par

Let $U, W$ be in ${\mathcal{U}}_{_0}$ such that $W\circ W\subseteq U$,
$(\xi^{\prime},\xi)\in \overline{W}$ and
$(\xi,\xi^{\prime\prime})\in \overline{W}$. We will prove that
$(\xi^{\prime},\xi^{\prime\prime})\in \overline{U}$. The only
interesting case is if ${{\mathcal{A}}}_{\xi}\nsubseteq {{\mathcal{A}}}_{\xi^\prime}$ and ${{\mathcal{A}}}_{\xi^{\prime\prime}}\nsubseteq {{\mathcal{A}}}_{\xi}$. Then, there is a $\tau_{_p}$-net $(x_a)_{a\in A}\in {{\mathcal{A}}}_{\xi^\prime},$ whose the elements of a final segment fulfil
$\tau.((x_a)_a, (z_\gamma)_\gamma)\in W$ for every
$(z_\gamma)_{\gamma\in\Gamma}\in {{\mathcal{A}}}_{\xi}$. In a similar
process, there is a final segment of a $\tau_{_p}$-net
$(z_{\gamma})_{{\gamma}\in \Gamma}\in {{\mathcal{A}}}_{\xi}$ such that
$\tau.((z_{\gamma})_{\gamma},(y _\beta)_{\beta})\in W,$  for every
$(x_{_\beta})_{_{\beta\in B}} \in {{\mathcal{A}}}_{\xi^{\prime\prime}}.$ In
conclusion
 $\tau.((x_a)_a,(x_{_{\beta}})_{_{\beta}})\in W\circ
W\subseteq U$ and this completes the proof.
\end{proof}
\bigskip
\par\noindent

\begin{theorem}\label{a17}{\rm For each $U\in {\mathcal{U}}_{_0}$,
$(\phi(a),\phi(b))\in \overline{U}$ if and only if $(a,b)\in U$.}
\end{theorem}
\begin{proof}Let $U\in {\mathcal{U}}_{_0}$ and $(a,b)\in U$. Then, there is a
$W\in {\mathcal{U}}_{_0}$ such that $W^{-1}(a)\times W(b)\subseteq U$.
Thus, for every $\tau_{_p}$-net $(x_{_\beta})_{_{\beta\in B}}$ converging
to $b$, we have $\tau.(a,(x_{_\beta})_{_\beta})\in U$ and since $a$
is a $\tau_{_p}$-net of $\phi(a)$, we conclude that $(\phi(a),\phi(b))\in
\overline{U}$.
\par
Conversely: let $(\phi(a),\phi(b))\in \overline{U}$. Then there is a
$\tau_{_p}$-net, say $(x_a)_{a\in A}\in {{\mathcal{A}}}_{_{\phi(a)}}$, such that for every
$(y_{_{\beta}}\!\!^{^j})_{_{\beta\in B_j}}\in {{\mathcal{A}}}_{_{\phi(b)}}$
there holds $\tau.((x_{a})_a,(y_{_{\beta}}\!\!^{^j})_{_\beta})\in U$.
Since $b$ is a $\tau_{_p}$-net of $\phi(b)$ we have
$\tau.((x_{_a})_{_a},b)\in U$. Since $U\in {\mathcal{U}}_{_0}$ there is a
$W\in {\mathcal{U}}_{_0}$ such that $W^{-1}(x_{a})\times W(b)\subseteq
U$, for all the elements $x_{a}$ of a final segment of
$(x_{a})_{{a\in A}}$. And since $(x_{a})_{{a\in A}}$ converges
to $a$, $a\in W^{-1}(x_a)$ and finally $(a,b)\in U$.
\end{proof}

\par
\bigskip
\par\noindent
We also have the following:
\par\noindent
\begin{proposition}\label{a18}{\rm  If $(x_a)_{a\in A}$ is a $\tau_{_p}$-net of a $\xi\in
\overline{X}$, then $\lim\limits_a \phi(x_a)=\xi$. Dually, if
$(y_{_\beta})_{_{\beta\in B}}$ is a $\tau_{_q}$-conet of a $\xi\in
\overline{X}$, then $\lim\limits_\beta
(\phi(y_{_\beta}),\xi)=0$.}
\end{proposition}
\begin{proof}Let $V, U$\ in ${\mathcal{U}}_{_0}$ such that $V\circ V\subseteq
U$. We have that $(x_{_\gamma},x_{_a})\in V$, for each $a,\gamma$ in
$A$ with $\gamma\geq a\geq a_{_V}$ ($a_{_V}$ the extreme index of
$(x_a)_{a\in A}$ for $V$). Fix an $a\geq a_{_V}$ and pick a
$\tau_{_p}$-net $(x_{_\delta}^\kappa)_{_{\delta\in \Delta_\kappa}}$ of
$\phi(x_a)$. Then $x_{_\delta}^\kappa\longrightarrow x_a$ and so
$(x_a,x_{_\delta}^\kappa)\in V$, whenever $\delta\geq
\delta_{_0}(\kappa$) for some $\delta_{_0}(\kappa)\in
\Delta_\kappa$. Hence, $(x_{_\gamma},x_{_\delta}^\kappa)\in U$ for
$\gamma\geq a$ and $\delta\geq \delta_{_0}(\kappa)$. Hence
$(\xi,\phi(x_a))\in \overline{U}$, whenever $a\geq a_{_V}$.
\par
The proof of the dual is similar.
\end{proof}

\begin{corollary}\label{a19}{\rm The set $\phi(X)$ is dense in $(\overline{X},\overline{\mathcal{U}})$
and $(\overline{X},(\overline{\mathcal{U}})^{-1})$.}
\end{corollary}
\begin{proof}The result comes directly from the above
proposition, since, if $(x_{_a})_{_{a\in A}}\in {\mathcal{A}}_{_{\xi}}$ for a $\xi\in \overline{X}\setminus \phi(X)$,
then $\displaystyle\lim_a\phi(x_{_a})=\xi$; thus, for any
$U\in {\mathcal{U}}_{_0}$, there are some $\phi(x_{_a})$ which belong to
$\overline{U}(\xi)$.
\end{proof}

Thus, the map $\phi: X\rightarrow \overline{X}$, apart of being a
set theoretical embedding, is a topological embedding as well.

\begin{definition}\label{a20} {\rm We call the structure
$(\overline{X},\overline{{\mathcal{U}}})$, {\it the}
$\tau$-{\it complement of} $(X,{\mathcal{U}})$ and the whole process of
this construction {\it the} $\tau$-{\it completion of} $(X,{\mathcal{U}})$.}
\end{definition}

\par
\smallskip
\par

The more hard point of the $\tau$-completion theory is the proof
that the structure $(\overline{X},\overline{{\mathcal{U}}})$ is
$\tau$-complete. We follow the usual procedure: we consider a
$\tau_{_p}$-net $({\xi}_{a^\star})_{a^\star\in {A}^\star}$ (and
one of its $\tau_{_q}$-conets $({\eta}_{\beta^\star})_{\beta^\star\in {B}^\star}$) in
$(\overline{X},\overline{{\mathcal{U}}})$ and we suppose that the
$\tau_{_p}$-nets do not converge, which means that they have not any {\it
end point}. So we define a suitable $\tau$-cut in $(X,{\mathcal{U}})$,
firstly a $\tau$-cut depended upon an entourage (Lemma \ref{a21}). Next we prove
(Lemma \ref{a23}) that, independently of this entourage a point of $\overline{X}$,
say $\xi$, is
fixed and that $({\xi}_{a^\star})_{a^\star\in {A}^\star}$
converges to $\xi$. In the Lemmas \ref{a22} and \ref{a24} we refer to
some $\tau_{_q}$-conets and, lastly, in theTheorem \ref{a26} we prove the main result:
there is a $\tau$-completion for any $T_{_0}$ quasi-uniform space.

\par
\smallskip

\section{The basic lemmas for the construction of a $\tau$-completion.}

\par
\smallskip
\par\noindent
The quasi-uniformity ${\mathcal{U}}_{_0}$ always is the referred in the
beginning of the \S 2. For brevity and simplicity we make use of
some phrases and some notation.

\par\smallskip\par
(1) If $({\xi}_{a^\star})_{a^\star\in
{A}^\star}$ in $\overline{X}$
is a $\tau_{_p}$-net in $\overline{X}$, then
\begin{center} $(x_{\rho(\kappa_{a^{\star}})})_
{\rho(\kappa_{a^{\star}}) \in P_{(\kappa_{a^{\star}})}} $
\end{center}
will denote a $\tau_{_p}$-net in $\mathcal{A}_{_{{\xi}_{a^\star}}}$ for a concrete $a^{\star}\in A^{\star}$.
More precisely, if $a^{\star}$ is a fixed
index of $A^{\star}$, then
$\kappa_{a^{\star}}$ denote the different $\tau_{_p}$-nets of
$\xi_{a^{\star}}$ and $\rho(\kappa_{a^{\star}})$ denote the index set of each
$\kappa_{a^{\star}}$ $\tau_{_p}$-net. Finally,
$\rho_{_V}(\kappa_{a^{\star}})$ denotes the extreme index of $(x_{\rho(\kappa_{a^{\star}})})_
{\rho(\kappa_{a^{\star}}) \in P_{(\kappa_{a^{\star}})}} $
for $V$ (see fig. 10).

\bigskip\bigskip\bigskip
\par

\pspicture(0,0)(0,0)

\Rput[ul](3,-2.15){${\mathcal{A}}_{_{\xi_{_{a^{\star}}}}}$}

\Rput[ul](7.3,-2.15){${\mathcal{B}}_{_{\xi_{_{a^{\star}}}}}$}

\pscircle[linewidth=2pt](5.43,-1.30){0.2}

\psline[linestyle=dotted]{-}(3.9,-1.34)(5.3,-1.3)

\qdisk(.9,-1.1){1pt}

\qdisk(1.1,-1.1){1pt}

\qdisk(1.3,-1.1){1pt}

\qdisk(1.5,-1.05){1pt}

\qdisk(1.7,-1.05){1pt}

\Rput[ul](1.65,-1.1){$\star$}

\psline{->}(1.3,-1.7)(1.82,-1.15)

\Rput[ul](0.65,-1.9)
{$x_{\rho_{_V}(\kappa_{a^{\star}})}$}

\qdisk(2.1,-1.2){1pt}

\qdisk(2.3,-1.1){1pt}

\qdisk(2.5,-1.2){1pt}

\qdisk(2.7,-1.3){1pt}

\qdisk(2.9,-1.4){1pt}

\qdisk(3.1,-1.35){1pt}

\qdisk(3.3,-1.37){1pt}

\qdisk(3.5,-1.31){1pt}

\qdisk(3.7,-1.34){1pt}

\psline{->}(5.4,-1.17)(5.4,-.2)

\Rput[ul](5.5,-.3){$\xi_{_{a^{\star}}}$}

\qdisk(9.8,-.6){1pt}

\qdisk(9.5,-.7){1pt}

\qdisk(9.2,-.7){1pt}

\qdisk(8.9,-.6){1pt}

\qdisk(8.7,-.8){1pt}

\qdisk(8.5,-.56){1pt}

\qdisk(8.3,-.8){1pt}

\qdisk(8.1,-.78){1pt}

\qdisk(7.9,-.91){1pt}

\qdisk(7.7,-.95){1pt}

\qdisk(7.5,-.9){1pt}

\qdisk(7.3,-1){1pt}

\psline[linestyle=dotted]{-}(7.3,-1)(5.7,-1.3)

\qdisk(9.8,-1.4){1pt}

\qdisk(9.5,-1.5){1pt}

\qdisk(9.2,-1.45){1pt}
\qdisk(8.9,-1.5){1pt}
\qdisk(8.7,-1.2){1pt}
\qdisk(8.5,-1.46){1pt}
\qdisk(8.3,-1.42){1pt}

\qdisk(8.1,-1.38){1pt}

\qdisk(7.9,-1.28){1pt}

\qdisk(7.7,-1.35){1pt}

\qdisk(7.5,-1.28){1pt}

\qdisk(7.3,-1.32){1pt}

\psline[linestyle=dotted]{-}(7.3,-1.32)(5.7,-1.3)

\psline{<-}(1.7,-3.1)(5.3,-3.1)

\Rput[ul](5.25,-3.1){${\rm V}$}

\psline{->}(5.8,-3.1)(9.4,-3.1)

\psline{<-}(.1,-.1)(.1,-1)

\Rput[ul](-.15,-1.2){$\kappa_{a^{\star}}$}

\psline{->}(.1,-1.4)(.1,-2.2)

\Rput[ul](.5,-.5){$(x_{\rho(\kappa_{a^{\star}})})_
{\rho(\kappa_{a^{\star}}) \in P_{(\kappa_{a^{\star}})}}$}

\endpspicture

\par\bigskip\bigskip\bigskip\bigskip\bigskip\bigskip\bigskip\bigskip\bigskip\bigskip\bigskip\bigskip\bigskip
\begin{center}
Figure 10
\end{center}
\par\bigskip\smallskip\par

\par\bigskip\smallskip\par

(2) Let $I_{_{a_{_0}}}$ (resp. $I_{_{\beta_{_0}}}$) denote the final segment of $(x_{_a})_{_{a\in A}}$
(resp. $(y_{_\beta})_{_{\beta\in B}}$) with initial element $x_{_{a_{_0}}}$ (resp. $y_{_{\beta_{_0}}}$).
We say that the pair $(I_{_{\beta_{_0}}},I_{_{a_{_0}}})$ is {\it $W$-close}
if $(y_{_\beta},x_{_a})\in W$ whenever $\beta\geq \beta_{_0}$ and $a\geq a_{_0}$.
By analogy we
face the same problems for $\tau_{_q}$-conets.

We begin with a $\tau_{_p}$-net $({\xi}_{a^\star})_{a^\star\in
{A}^\star}$ in $\overline{X}$ and we advance in the construction of
the demanded $\tau_{_p}$-{\it net for} $W$ by transfinite induction on a
well ordered set which, without loss of generality, may be
considered as a subnet of $A^{\star}$.

\begin{lemma}\label{a21}{\rm Let $({\xi}_ {a^\star})_{a^\star\in
{A}^\star}$ be a $\tau_{_p}$-net in $(\overline{X},\overline{\mathcal{U}})$
without end point such that for each $\gamma^{\star}>a^{\star}$,
${{\mathcal{A}}}_{_{\xi_{_{a^{\star}}}}}\nsubseteq {{\mathcal{A}}}_{_{\xi_{_{\gamma^{\star}}}}}$. Let also $W\in {\mathcal{U}}$ be an
entourage in ${\mathcal{U}}_{_0}$. Then, there is a $\tau_{_p}$-net
$(t_{_\lambda})_{_{\lambda\in \Lambda}}$
such that the $\tau_{_p}$-nets $({\xi}_{a^\star})_{a^\star\in
{A}^\star}$ and $(\phi(t_{_\lambda})){_{_{\lambda\in \Lambda}}}$ are
left cofinal for $\overline{W}$.}
\end{lemma}

\begin{proof} Let $(\xi)=({\xi}_ {a^\star})_{a^\star\in {A}^\star}$ be as above, $W, V$ be entourages of
${\mathcal{U}}_{_0}$ such that $V\circ V\circ V\subseteq W$ and
$\xi_{a^{\star}_{_{\overline{V}}}}$ be the extreme point of
$(\xi)$ for $\overline{V}$. We symbolize by

\begin{center} $(x_{\rho(\kappa_{a^{\star}})})_
{\rho(\kappa_{a^{\star}}) \in P_{_{\!\kappa_{a^{\star}}}}} $
\end{center}
\par\noindent
any $\tau_{_p}$-net of $\xi_{a^{\star}}$,
with the evident meaning of the notation and by
$\rho_{_V}(\kappa_{a^{\star}})$
the extreme index of
$(x_{\rho(\kappa_{a^{\star}})})_
{\rho(\kappa_{a^{\star}}) \in P_{_{\!\kappa_{a^{\star}}}}}$ for $V$. The entourage $V$ is fixed during all the lemma's proof.
\par
We advance to the construction of the demanded ${\tau_{_p}}$-{\it net
for} $W$ by transfinite induction on a well ordered set
subnet of
$A^{\star}$.
\bigskip
\par

(1) {\it The first step.}
\par
Let $\xi_{{a_{_0}^{\star}}}$ in $(\xi)$, $a^{\star}_0>
a^{\star}_{_{\overline{V}}}$ and $(x_{\rho({\kappa}\!^{^{_V}}\
\!\!\!\!\!\!\!_{a\!^{\star}_{_0}})})_ {\rho({\kappa}\!^{^{_V}}\
\!\!\!\!\!\!\!_{a\!^{\star}_{_0}}) \in P_{_{{\!\kappa}\!^{^{_V}}\
\!\!\!\!\!\!\!_{a\!^{\star}_{_0}}}}}$ be an arbitrary $\tau_{_p}$-net of
${{\mathcal{A}}}_{{_\xi}_{_{a^{\star}_0}}}$.

Put ${I}_{{{\kappa}\!^{^{_V}}\
\!\!\!\!\!\!\!_{a\!^{\star}_{_0}}}}=\{x_{\rho({\kappa}\!^{^{_V}}\
\!\!\!\!\!\!\!_{a\!^{\star}_{_0}})}\vert
 \rho({\kappa}\!^{^{_V}}\  \!\!\!\!\!_{a\!^{\star}_{_0}})\geq
 \rho_{_V}({\kappa}\!^{^{_V}}\  \!\!\!\!\!_{a\!^{\star}_{_0}}) \}$ and
$G_{_{ a^{\star}_{_{0}}}}={I}_{{{\kappa}\!^{^{_V}}\
\!\!\!\!\!\!\!_{a\!^{\star}_{_0}}}}$.

\par
Since $(\xi)$ has not a last element, there are
$a^{\star}_1>\gamma^{\star}>a^\star_0$ ($a^{\star}_{_1}$ assigns to
$\xi_{a^\star_{_1}}$). Since
$(\xi_{_{a^{\star}_{_1}}},\xi_{_{\gamma^{\star}}})\in \overline{V}$
and $(\xi_{_{\gamma^{\star}}},\xi_{_{a^{\star}_0}})\in \overline{V}$
we have the following:
\par\noindent
(A) There is a final segment for $V$
\begin{center}${I}_{{{\kappa}\!^{^{_V}}\  \!\!\!\!\!\!\!_{a\!^{\star}_{_1}}}}=\{x_{\rho({\kappa}\!^{^{_V}}\
\!\!\!\!\!\!\!_{a\!^{\star}_{_1}})}\vert
 \rho({\kappa}\!^{^{_V}}\  \!\!\!\!\!_{a\!^{\star}_{_1}})\geq
 \widetilde{\rho}({\kappa}\!^{^{_V}}\  \!\!\!\!\!_{a\!^{\star}_{_1}}) \}$
with $\widetilde{\rho}({\kappa}\!^{^{_V}}\
\!\!\!\!\!_{a\!^{\star}_{_1}})\geq \rho_{_V}({\kappa}\!^{^{_V}}\
\!\!\!\!\!_{a\!^{\star}_{_1}})$
 \end{center}
\par\noindent
of a $\tau_{_p}$-net $(x_{\rho({\kappa}\!^{^{_V}}\
\!\!\!\!\!\!\!_{a\!^{\star}_{_1}})})_{\rho({\kappa}\!^{^{_V}}\
\!\!\!\!\!\!\!_{a\!^{\star}_{_1}}) \in P_{({\kappa}\!^{^{_V}}\
\!\!\!\!\!\!\!_{a\!^{\star}_{_1}})}} \in {{\mathcal{A}}}_{{_\xi}_{_{a^{\star}_1}}}$ for which there
are final segments $I_{_{k_{_{\gamma^{\star}}}}}$ of all $\tau_{_p}$-nets
of $\xi^\star_{\gamma^{\star}}$ such that the pair
$({I}_{{{\kappa}\!^{^{_V}}\
\!\!\!\!\!\!\!_{a\!^{\star}_{_1}}}},I_{_{k_{_{\gamma^{\star}}}}})$
is $V$-close.\footnote{In the following, the
entourage $V$ in the form ${\kappa}\!^{^{_V}}\
\!\!\!\!\!_{a\!^{\star}_{_i}}$ will indicate the concrete $\tau_{_p}$-net
which have been chosen in the point
$\xi^{\star}_{_{a^{\star}_{_i}}}$ satisfying (A).}
\par\noindent
(B) There is a final segment for $V$
\begin{center} $I_{_{k^{^0}_{_{\gamma^{\star}}}}}=
\{x_{\rho(k^{^0}_{_{\gamma^{\star}}})}\vert
\rho(k^{^0}_{_{\gamma^{\star}}})\geq
\widetilde{\rho}(k^{^0}_{_{\gamma^{\star}}})\}$
\end{center}
\par\noindent
of a $\tau_{_p}$-net $(x_{\rho(k^{^0}_{_{\gamma^{\star}}})})_
{\rho(k^{^0}_{_{\gamma^{\star}}}) \in
P_{(k^{^0}_{_{\gamma^{\star}}})}} \in {{\mathcal{A}}}_{{_\xi}_{_{\gamma^{\star}}}}$ for which there are final segments
$\hat{I}_{_{k_{_{a_{_0}^{\star}}}}}$ of all $\tau_{_p}$-nets of
$\xi_{_{a^{\star}_0}}$ such that the pair
$(I_{_{k^{^0}_{_{\gamma^{\star}}}}},\hat{I}_{_{k_{_{a_{_0}^{\star}}}}})$
is $V$-close.
\par\noindent
(C) Since $(x_{\rho({\kappa}\!^{^{_V}}\
\!\!\!\!\!\!\!_{a\!^{\star}_{_0}})})_ {\rho({\kappa}\!^{^{_V}}\
\!\!\!\!\!\!\!_{a\!^{\star}_{_0}}) \in P_{({\kappa}\!^{^{_V}}\
\!\!\!\!\!\!\!_{a\!^{\star}_{_0}})}}$ is $\tau_{_p}$-net we have that
\begin{center}
$\hat{I}_{{{\kappa}\!^{^{_V}}\
\!\!\!\!\!\!\!_{a\!^{\star}_{_0}}}}\subseteq
{I}_{{{\kappa}\!^{^{_V}}\ \!\!\!\!\!\!\!_{a\!^{\star}_{_0}}}}$ \ or
${I}_{{{\kappa}\!^{^{_V}}\
\!\!\!\!\!\!\!_{a\!^{\star}_{_0}}}}\subseteq
\hat{I}_{{{\kappa}\!^{^{_V}}\  \!\!\!\!\!\!\!_{a\!^{\star}_{_0}}}}$.
\end{center}
\par
Finally, from (A) and (B) for
$k_{_{\gamma^{\star}}}=k^{^0}_{_{\gamma^{\star}}}$
we conclude that:
\par\noindent
(D) a) The pair $({I}_{{{\kappa}\!^{^{_V}}\
\!\!\!\!\!\!\!_{a\!^{\star}_{_1}}}},\hat{I}_{_{k_{_{a_{_0}^{\star}}}}})$
is $V\circ V$-close for all $k_{_{
a^{\!\star}_{_0}}}$, which jointly with (C) for
$k_{_{a_{_0}^{\star}}}={\kappa}\!^{^{_V}}\
\!\!\!\!\!_{a\!^{\star}_{_0}}$ we have that:
\par
\bigskip
\par
\ b)\ The pair $({I}_{{{\kappa}\!^{^{_V}}\
\!\!\!\!\!\!\!_{a\!^{\star}_{_1}}}},{I}_{{{\kappa}\!^{^{_V}}\
\!\!\!\!\!\!\!_{a\!^{\star}_{_0}}}})$ is $V\circ V\circ V$-close (see fig.11).
\par

We put $G_{_{ a^{\star}_{_{1}}}}=G_{_{ a^{\star}_{_{0}}}}\cup
{I}_{{{\kappa}\!^{^{_V}}\ \!\!\!\!\!\!\!_{a\!^{\star}_{_1}}}}$.
Since $(\xi)$ has not a last element, there are $a^{\star}_2, \gamma^{\star\star}\in A^{\star}$ such that
$a^{\star}_2>\gamma^{\star\star}>a^\star_1>a^{\star}_0$.

\par

\bigskip\bigskip\bigskip
\par

\pspicture(0,0)(0,0)

\pscircle[linewidth=2pt](1.33,-1.30){0.2}

\pscircle[linewidth=2pt](3.33,-1.30){0.2}

\pscircle[linewidth=2pt](5.33,-1.30){0.2}

\pscircle[linewidth=2pt](8.33,-1.30){0.2}

\multiput(1.94,-1.52)(1.46,-1.30){1}{*}

\multiput(2.24,-1.52)(1.46,-1.30){1}{*}

\multiput(2.54,-1.52)(1.46,-1.30){1}{*}

\multiput(4.64,-1.52)(1.46,-1.30){1}{*}

\multiput(4.34,-1.52)(1.46,-1.30){1}{*}

\multiput(4.04,-1.52)(1.46,-1.30){1}{*}

\multiput(6.94,-1.52)(1.46,-1.30){1}{*}

\multiput(6.64,-1.52)(1.46,-1.30){1}{*}

\multiput(6.34,-1.52)(1.46,-1.30){1}{*}

\multiput(9.64,-1.52)(1.46,-1.30){1}{*}

\multiput(9.34,-1.52)(1.46,-1.30){1}{*}

\multiput(9.1,-1.52)(1.46,-1.30){1}{*}

\psline{->}(9.9,-1.30)(10.3,-1.30)

\multiput(.2,-1.52)(1.46,-1.30){1}{*}

\multiput(.5,-1.52)(1.46,-1.30){1}{*}

\multiput(.8,-1.52)(1.46,-1.30){1}{*}

\Rput[ul](.1,-1){$(\xi)$}

\Rput[ul](1,-.6){$\xi_{_{a^{\star}_{_V}}}$}

\Rput[ul](5,-.6){$\xi_{_{\gamma^{\star}}}$}

\Rput[ul](3,-.6){$\xi_{_{a^{\star}_{_0}}}$}

\Rput[ul](8,-.6){$\xi_{_{a^{\star}_{_1}}}$}

\psline[linestyle=dotted]{-}(3.1,-2.44)(3.3,-1.5)

\qdisk(3.1,-2.56){1pt}

\qdisk(3.1,-2.76){1pt}

\qdisk(3,-2.86){1pt}

\qdisk(3.15,-2.97){1pt}

\qdisk(3.1,-3.08){1pt}

\qdisk(3.07,-3.18){1pt}

\Rput[ul](7.7,-3.35){$\star$}

\Rput[ul](5.95,-3.36){$x_{\widetilde{\rho}({\kappa}\!^{^{_V}}\
\!\!\!\!\!\!\!_{a\!^{\star}_{_1}})}$}

\Rput[ul](2.7,-3.35){$\star$}

\qdisk(3.1,-3.38){1pt}

\qdisk(3.07,-3.58){1pt}

\qdisk(2.95,-3.75){1pt}

\Rput[ul](2.8,-3.95){$\star$}

\psline{->}(2.2,-4.7)(2.9,-4.15)

\Rput[ul](.8,-4.36){$x_{\rho_{_V}({\kappa}\!^{^{_V}}\
\!\!\!\!\!\!\!_{a\!^{\star}_{_0}})}$}

\psline{->}(7.2,-3.5)(7.9,-3.4)

\psline{->}(2.2,-3.5)(2.9,-3.4)

\Rput[ul](.8,-3.36){$x_{\widetilde{\rho}({\kappa}\!^{^{_V}}\
\!\!\!\!\!\!\!_{a\!^{\star}_{_0}})}$}

\Rput[ul](-.3,-2.46){$
(x_{\rho({\kappa}\!^{^{_V}}\
\!\!\!\!\!\!\!_{a\!^{\star}_{_0}})})_ {\rho({\kappa}\!^{^{_V}}\
\!\!\!\!\!\!\!_{a\!^{\star}_{_0}}) \in P_{({\kappa}\!^{^{_V}}\
\!\!\!\!\!\!\!_{a\!^{\star}_{_0}})}}$}

\psline{<-}(3.8,-1.30)(3.63,-2.4)

\Rput[ul](3.35,-2.8){${I}_{{{\kappa}\!^{^{_V}}\
\!\!\!\!\!\!\!_{a\!^{\star}_{_0}}}}$}

\psline{->}(3.5,-3)(3.3,-4.1)

\psline[linestyle=dotted]{-}(8.1,-2.44)(8.3,-1.5)

\qdisk(8.1,-2.56){1pt}

\qdisk(8.1,-2.76){1pt}

\qdisk(8,-2.86){1pt}

\qdisk(8.15,-2.97){1pt}

\qdisk(8.1,-3.08){1pt}

\qdisk(8.07,-3.18){1pt}

\qdisk(8.1,-3.38){1pt}

\qdisk(8.07,-3.58){1pt}

\qdisk(7.95,-3.75){1pt}

\Rput[ul](7.8,-3.95){$\star$}

\psline{->}(7.2,-4.7)(7.9,-4.15)

\Rput[ul](6,-4.36){$
x_{\widetilde{\rho}({\kappa}\!^{^{_V}}\
\!\!\!\!\!\!\!_{a\!^{\star}_{_1}})}$}

\Rput[ul](4.7,-2.46){$
(x_{\rho({\kappa}\!^{^{_V}}\
\!\!\!\!\!\!\!_{a\!^{\star}_{_1}})})_ {\rho({\kappa}\!^{^{_V}}\
\!\!\!\!\!\!\!_{a\!^{\star}_{_1}}) \in P_{({\kappa}\!^{^{_V}}\
\!\!\!\!\!\!\!_{a\!^{\star}_{_1}})}}$}

\psline{<-}(8.8,-1.30)(8.63,-2.4)

\Rput[ul](8.35,-2.8){${I}_{{{\kappa}\!^{^{_V}}\
\!\!\!\!\!\!\!_{a\!^{\star}_{_1}}}}$}

\psline{->}(8.5,-3)(8.3,-4.1)

\endpspicture

\par\bigskip\bigskip\bigskip
\bigskip\bigskip\bigskip\bigskip\bigskip\bigskip\bigskip\bigskip\bigskip\bigskip\bigskip\bigskip
\begin{center}
Figure 11
\end{center}

\par\noindent
(E) From $a^{\star}_2>\gamma^{\star\star}>a^\star_1$ and
$a^{\star}_2>\gamma^{\star\star}>a^{\star}_0$ according to the above process consisting of four steps
((A)$\rightarrow$ (D)) we conclude that:

\par
\smallskip
\par\noindent
There is a final segment

\begin{center}${I}_{{{\kappa}\!^{^{_V}}\  \!\!\!\!\!\!\!_{a\!^{\star}_{_2}}}}=\{x_{\rho({\kappa}\!^{^{_V}}\
\!\!\!\!\!\!\!_{a\!^{\star}_{_2}})}\vert
\rho({\kappa}\!^{^{_V}}\  \!\!\!\!\!_{a\!^{\star}_{_2}})\geq
\widetilde{\rho}({\kappa}\!^{^{_V}}\  \!\!\!\!\!_{a\!^{\star}_{_2}}) \}$
with $\widetilde{\rho}({\kappa}\!^{^{_V}}\
\!\!\!\!\!_{a\!^{\star}_{_2}})\geq \rho_{_V}({\kappa}\!^{^{_V}}\
\!\!\!\!\!_{a\!^{\star}_{_2}})$
\end{center}
\par\noindent
of a $\tau_{_p}$-net $(x_{\rho({\kappa}\!^{^{_V}}\
\!\!\!\!\!\!\!_{a\!^{\star}_{_2}})})_{\rho({\kappa}\!^{^{_V}}\
\!\!\!\!\!\!\!_{a\!^{\star}_{_2}}) \in P_{({\kappa}\!^{^{_V}}\
\!\!\!\!\!\!\!_{a\!^{\star}_{_2}})}} \in {{\mathcal{A}}}_{{_\xi}_{_{a^{\star}_2}}}$, such that:
\par
\smallskip
\par\noindent
a) The pair $({I}_{{{\kappa}\!^{^{_V}}\
\!\!\!\!\!\!\!_{a\!^{\star}_{_{i}}}}},\hat{I}_{_{k_{_{a_{_j}^{\star}}}}})$
is $V\circ V$-close for all
$k_{_{a_{_j}^{\star}}}$, $i,j\in \{0,1,2\}$, $i>j$ (for
each $j\in \{0,1\}$, $\hat{I}_{_{k_{_{a_{_j}^{\star}}}}}$ are the
final segments of the $k_{_{a_{_j}^{\star}}}$ $\tau_{_p}$-net of the
point $\xi_{_{a_{_j}^{\star}}}$ which we take by applying the step
(B) in the relations
$(\xi_{_{\gamma^{\star}}},\xi_{_{a^{\star}_1}})\in \overline{V}$ and
$(\xi_{_{\gamma^{\star}}},\xi_{_{a^{\star}_0}})\in \overline{V}$).

\par
\bigskip
\par\noindent
b)\ The pair $({I}_{{{\kappa}\!^{^{_V}}\
\!\!\!\!\!\!\!_{a\!^{\star}_{_i}}}},{I}_{{{\kappa}\!^{^{_V}}\
\!\!\!\!\!\!\!_{a\!^{\star}_{_j}}}})$ is $V\circ V\circ V$-close for
each $i,j\in \{0,1,2\}$, $i>j$.

\par
\bigskip
\par

\par
We put
\par
\smallskip
\par\noindent
(F) \ \ $G_{_{ a^{\star}_{_{2}}}}=G_{_{ a^{\star}_{_{1}}}}\cup
{I}_{{{\kappa}\!^{^{_V}}\ \!\!\!\!\!\!\!_{a\!^{\star}_{_2}}}}$ (see fig. 12).

\pspicture(0,0)(0,0)

\pscircle[linewidth=2pt](1.43,-1.30){0.2}

\pscircle[linewidth=2pt](3.33,-1.30){0.2}

\pscircle[linewidth=2pt](5.23,-1.30){0.2}

\pscircle[linewidth=2pt](7.13,-1.30){0.2}

\pscircle[linewidth=2pt](9.03,-1.30){0.2}

\multiput(1.94,-1.52)(1.46,-1.30){1}{*}

\multiput(2.24,-1.52)(1.46,-1.30){1}{*}

\multiput(2.54,-1.52)(1.46,-1.30){1}{*}

\multiput(4.64,-1.52)(1.46,-1.30){1}{*}

\multiput(4.34,-1.52)(1.46,-1.30){1}{*}

\multiput(4.04,-1.52)(1.46,-1.30){1}{*}

\multiput(6.36,-1.52)(1.46,-1.30){1}{*}

\multiput(6.08,-1.52)(1.46,-1.30){1}{*}

\multiput(5.80,-1.52)(1.46,-1.30){1}{*}

\multiput(8.25,-1.52)(1.46,-1.30){1}{*}

\multiput(7.94,-1.52)(1.46,-1.30){1}{*}

\multiput(7.65,-1.52)(1.46,-1.30){1}{*}

\multiput(9.42,-1.52)(1.46,-1.30){1}{*}

\multiput(9.67,-1.52)(1.46,-1.30){1}{*}

\multiput(9.95,-1.52)(1.46,-1.30){1}{*}

\psline{->}(10.2,-1.30)(10.5,-1.30)

\multiput(.2,-1.52)(1.46,-1.30){1}{*}

\multiput(.5,-1.52)(1.46,-1.30){1}{*}

\multiput(.8,-1.52)(1.46,-1.30){1}{*}

\Rput[ul](.1,-1){$(\xi)$}

\Rput[ul](1,-.6){$\xi_{_{a^{\star}_{_0}}}$}

\Rput[ul](3,-.6){$\xi_{_{\gamma^{\star}}}$}

\Rput[ul](5,-.6){$\xi_{_{a^{\star}_{_1}}}$}

\Rput[ul](6.85,-.6){$\xi_{_{\gamma^{\star\star}}}$}

\Rput[ul](8.7,-.6){$\xi_{_{a^{\star}_{_2}}}$}

\psline{-}(1,-2.8)(1.37,-1.40)

\Rput[ul](.74,-2.8){$\star$}

\Rput[ul](-.7,-2){$G_{_{ a^{\star}_{_{0}}}}={I}_{{{\kappa}\!^{^{_V}}\
\!\!\!\!\!\!\!_{a\!^{\star}_{_0}}}}$}

\psline{-}(4.8,-2.8)(5.17,-1.40)

\Rput[ul](4.54,-2.8){$\star$}

\Rput[ul](4.19,-2){${I}_{{{\kappa}\!^{^{_V}}\
\!\!\!\!\!\!\!_{a\!^{\star}_{_1}}}}$}

\psline{-}(8.6,-2.8)(8.97,-1.40)

\Rput[ul](8.34,-2.8){$\star$}

\Rput[ul](8,-2){${I}_{{{\kappa}\!^{^{_V}}\
\!\!\!\!\!\!\!_{a\!^{\star}_{_2}}}}$}

\psline{<-}(.96,-3.3)(2.4,-3.3)

\Rput[ul](2.5,-3.3){$
G_{_{ a^{\star}_{_{1}}}}$}

\psline{->}(3.5,-3.3)(4.8,-3.3)

\psline{<-}(.96,-3.7)(4.3,-3.7)

\Rput[ul](4.28,-3.76){$
G_{_{ a^{\star}_{_{2}}}}$}

\psline{->}(5.1,-3.7)(8.5,-3.7)

\endpspicture

\par\bigskip\bigskip\bigskip\bigskip\bigskip\bigskip\bigskip\bigskip
\bigskip\bigskip\bigskip\bigskip\bigskip
\begin{center}
Figure 12
\end{center}
\par\bigskip\smallskip\par

\par
(2) {\it From $\beta$ to $\beta+1$.}
\par\noindent
We intend to pick up, by induction, a subnet
$(\xi_{_{a^{\star}_{_{i}}}})_{_{i\in I_{_W}}}$ of $(\xi)$
with the properties that have been generated in Step 1.
We assume that $\beta$ is a
regular ordinal
and that, for every $\xi_{a^{\star}_{i}}$, $i\leq\beta$,
we have already chosen:
(i) The final segments
$\hat{I}_{_{k_{_{a_{_{i}}^{\star}}}}}$ of all the $\tau_{_p}$-nets of each $\xi_{a^{\star}_{i}}$ and (ii) the concrete final
segment ${I}_{{{\kappa}\!^{^{_V}}\
\!\!\!\!\!\!\!_{a\!^{\star}_{_{i}}}}}$ which has fixed extreme
point, the point $x_{\widetilde{\rho}({\kappa}\!^{^{_V}}\
\!\!\!\!\!\!\!_{a\!^{\star}_{_{i}}})}$
($\widetilde{\rho}({\kappa}\!^{^{_V}}\
\!\!\!\!\!_{a\!^{\star}_{_{i}}})\geq \rho_{_V}({\kappa}\!^{^{_V}}\
\!\!\!\!\!_{a\!^{\star}_{_{i}}})$). We then have:
\par
\smallskip\noindent
($A_1$) a) For each $j<i\leq \beta$, the pair
$({I}_{{{\kappa}\!^{^{_V}}\
\!\!\!\!\!\!\!_{a\!^{\star}_{_{i}}}}},{\hat{I}}_{_{k_{_{a_{_{j}}^{\star}}}}})$
is $V\circ V$-close for all $k_{_{a_{_{j}}^{\star}}}$.
\par
\bigskip
\par
\ \ \ b)\  For each $j<i\leq \beta$, the pair
$({I}_{{{\kappa}\!^{^{_V}}\
\!\!\!\!\!\!\!_{a\!^{\star}_{_{i}}}}}$,${I}_{{{\kappa}\!^{^{_V}}\
\!\!\!\!\!\!\!_{a\!^{\star}_{_{j}}}}})$ is $V\circ V\circ V$-close.
\par\noindent
($B_1$) Everyone of these points $\xi_{a^{\star}_{i}}$
corresponds to another set
\begin{center}
$G_{_{a^{\star}_{_{i}}}}=(\displaystyle\bigcup_{j<i}G_{_{
a^{\star}_{_{j}}}})\cup {I}_{{{\kappa}\!^{^{_V}}\
\!\!\!\!\!\!\!_{a\!^{\star}_{_{i}}}}}$
\end{center}
\par
\smallskip
\par\noindent

\par
\bigskip
\par\noindent
Since $\xi_{a^{\star}_{_\beta}}$ is not last element of
$(\xi)$, there are elements $\gamma^{\star}$ and $\epsilon^{\star}$
of $A^{\star}$such that $a^{\star}\!\!\!_{_{\beta}}<
\gamma^{\star}<\epsilon^{\star}$. Hence, for each $i\leq \beta$ there holds $a^{\star}_{i}<
\gamma^{\star}<\epsilon^{\star}$. But then,
as in the above case (E), there is a concrete final segment
\begin{center} $I_{_{k\!\!^{^V}_{\epsilon^{\star}}}}=
\{x_{\rho(k\!^{^V}_{\epsilon^{\star}})}\vert \rho(k\!^{^V}_{\epsilon^{\star}})\geq
\widetilde{\rho}(k\!^{^V}_{\epsilon^{\star}})\}$ with
$\widetilde{\rho}(k\!^{^V}_{\epsilon^{\star}})\geq
\rho_{_V}(k\!^{^V}_{\epsilon^{\star}})$
\end{center}
\par\noindent
of a $\tau_{_p}$-net
$(x_{\rho(k\!^{^V}_{\epsilon^{\star}})})_
{\rho(k\!^{^V}_{\epsilon^{\star}}) \in
P_{(k\!^{^V}_{\epsilon^{\star}})}} \in {{\mathcal{A}}}_{{_\xi}_{_{\epsilon^{\star}}}}$ such that:
\par
\smallskip
\par\noindent
a)\ The pair
$(I_{_{k\!\!^{^V}_{\epsilon^{\star}}}},{\hat{I}}_{_{k_{_{a_{_{i}}^{\star}}}}})$
is $V\circ V$-close for each $i\leq \beta$ and all $k_{_{a_{_{i}}^{\star}}}$.
\par
\bigskip
\par\noindent
b)\ The pair
$(I_{_{k\!\!^{^V}_{\epsilon^{\star}}}},{I}_{{{\kappa}\!^{^{_V}}\
\!\!\!\!\!\!\!_{a\!^{\star}_{_{i}}}}})$ is $V\circ V\circ V$-close
for each $i\leq \beta$.
\par
\smallskip
\par
We put  ${I}_{{{\kappa}\!^{^{_V}}\
\!\!\!\!\!\!\!_{a\!^{\star}_{_{\beta+1}}}}}=I_{_{k\!^{^V}_{\epsilon^{\star}}}}$
and $G_{_{
a^{\star}_{_{{\beta+1}}}}}=(\displaystyle\bigcup_{i\leq\beta}G_{_{a^{\star}_{_{i}}}})\cup
{I}_{{{\kappa}\!^{^{_V}}\
\!\!\!\!\!\!\!_{a\!^{\star}_{_{\beta+1}}}}}$.
\par
\smallskip
\par
It is evident that the above properties ($A_1$) to ($B_1$) are
extended for each $i\leq \beta+1$ (see fig. 13).

\pspicture(0,0)(0,0)

\pscircle[linewidth=2pt](1.43,-1.30){0.2}

\pscircle[linewidth=2pt](3.33,-1.30){0.2}

\pscircle[linewidth=2pt](5.23,-1.30){0.2}

\pscircle[linewidth=2pt](7.13,-1.30){0.2}

\pscircle[linewidth=2pt](9.03,-1.30){0.2}

\multiput(1.94,-1.52)(1.46,-1.30){1}{*}

\multiput(2.24,-1.52)(1.46,-1.30){1}{*}

\multiput(2.54,-1.52)(1.46,-1.30){1}{*}

\multiput(4.64,-1.52)(1.46,-1.30){1}{*}

\multiput(4.34,-1.52)(1.46,-1.30){1}{*}

\multiput(4.04,-1.52)(1.46,-1.30){1}{*}

\multiput(6.36,-1.52)(1.46,-1.30){1}{*}

\multiput(6.08,-1.52)(1.46,-1.30){1}{*}

\multiput(5.80,-1.52)(1.46,-1.30){1}{*}

\multiput(8.25,-1.52)(1.46,-1.30){1}{*}

\multiput(7.94,-1.52)(1.46,-1.30){1}{*}

\multiput(7.65,-1.52)(1.46,-1.30){1}{*}

\multiput(9.42,-1.52)(1.46,-1.30){1}{*}

\multiput(9.67,-1.52)(1.46,-1.30){1}{*}

\multiput(9.95,-1.52)(1.46,-1.30){1}{*}

\psline{->}(10.2,-1.30)(10.5,-1.30)

\multiput(.2,-1.52)(1.46,-1.30){1}{*}

\multiput(.5,-1.52)(1.46,-1.30){1}{*}

\multiput(.8,-1.52)(1.46,-1.30){1}{*}

\Rput[ul](.1,-1){$(\xi)$}

\Rput[ul](1,-.6){$\xi_{_{a^{\star}_{_j}}}$}

\Rput[ul](3,-.6){$\xi_{_{a^{\star}_{_i}}}$}

\Rput[ul](5,-.6){$\xi_{_{a^{\star}_{_\beta}}}$}

\Rput[ul](6.85,-.6){$\xi_{_{a^{\star}_{_\gamma}}}$}

\Rput[ul](8.7,-.6){$\xi_{_{a^{\star}_{_\epsilon}}}$}

\psline{-}(1,-2.8)(1.37,-1.40)

\Rput[ul](.74,-2.8){$\star$}

\Rput[ul](.21,-2){${I}_{{{\kappa}\!^{^{_V}}\
\!\!\!\!\!\!\!_{a\!^{\star}_{_j}}}}$}

\psline{-}(3,-2.8)(3.37,-1.40)

\Rput[ul](2.74,-2.8){$\star$}

\Rput[ul](2.31,-2){${I}_{{{\kappa}\!^{^{_V}}\
\!\!\!\!\!\!\!_{a\!^{\star}_{_i}}}}$}

\psline{-}(4.8,-2.8)(5.17,-1.40)

\Rput[ul](4.54,-2.8){$\star$}

\Rput[ul](4.19,-2){${I}_{{{\kappa}\!^{^{_V}}\
\!\!\!\!\!\!\!_{a\!^{\star}_{_\beta}}}}$}

\psline{-}(8.6,-2.8)(8.97,-1.40)

\Rput[ul](8.34,-2.8){$\star$}

\Rput[ul](9,-2){$I_{_{k\!^{^V}_{\epsilon^{\star}}}}=
{I}_{{{\kappa}\!^{^{_V}}\
\!\!\!\!\!\!\!_{a\!^{\star}_{_{\beta+1}}}}}$}

\psline{<-}(.26,-3.3)(2.4,-3.3)

\Rput[ul](2.5,-3.3){$
G_{_{ a^{\star}_{_{\beta}}}}$}

\psline{->}(3.5,-3.3)(4.8,-3.3)

\psline{<-}(.26,-3.7)(4.3,-3.7)

\Rput[ul](4.28,-3.76){$
G_{_{ a^{\star}_{_{\beta+1}}}}$}

\psline{->}(5.1,-3.7)(8.5,-3.7)

\endpspicture

\par\bigskip\bigskip\bigskip\bigskip\bigskip\bigskip\bigskip\bigskip
\bigskip\bigskip\bigskip\bigskip\bigskip
\begin{center}
Figure 13
\end{center}
\par\bigskip\smallskip\par

\par

(3){\it The case of being $\beta$ a limit point.}
\par
\bigskip
\par\noindent
Let $\beta$ be a limit ordinal and $(\xi)$ be as above. We suppose
that we have constructed a subnet
$(\xi_{_{a_{_{i}}^{\star}}})_{i<\beta}$ of $(\xi)$ whose the
elements have indexes larger than $a^{\star}_{_{\overline{V}}}$ and
they constitute a linear subnet of $A^{\star}$. Moreover, for every
one of these points, say $\xi_{_{a_{_{i}}^{\star}}}$, it corresponds
the final segments ${\hat{I}}_{_{k_{_{a_{_{i}}^{\star}}}}}$ of all
the $k_{_{a_{_{i}}^{\star}}}$ $\tau_{_p}$-nets of
$\xi_{_{a_{_{i}}^{\star}}}$ as well as the concrete final segment
${I}_{{{\kappa}\!^{^{_V}}\ \!\!\!\!\!\!\!_{a\!^{\star}_{_{i}}}}}$
which has been chosen in the above process and it exclusively depends on $V$.
That final segments have
the following properties:
\par
\bigskip
\par\noindent
a$^{\prime}$) The pair $({I}_{{{\kappa}\!^{^{_V}}\
\!\!\!\!\!\!\!_{a\!^{\star}_{_i}}}},{\hat{I}}_{_{k_{_{a_{_j}^{\star}}}}})$
is $V\circ V$-close for each $j<i<\beta$ and all $k_{_{a_{_j}^{\star}}}$.
\par
\bigskip
\par\noindent
b$^{\prime}$) The pair $({I}_{{{\kappa}\!^{^{_V}}\
\!\!\!\!\!\!\!_{a\!^{\star}_{_i}}}},{I}_{{{\kappa}\!^{^{_V}}\
\!\!\!\!\!\!\!_{a\!^{\star}_{_j}}}})$ is $V\circ V\circ V$-close for
each $j<i<\beta$.

\par
\smallskip
\par\noindent
If a cofinal linear subset of $A^{\star}$ has ordinal number the
limit ordinal $\beta$, then the process is over. If it is not the
case, there are elements $\gamma^{\star}$ and $\delta^{\star}$ of
$A^{\star}$such that $a^{\star}\!\!\!_{_{\beta}}<
\gamma^{\star}<\delta^{\star}$. Hence $a^{\star}_i<
\gamma^{\star}<\delta^{\star}$ for each $i<\beta$. But then, as in
the above case (E), there is a final segment

\begin{center} $I_{_{k\!\!^{^V}_{\delta^{\star}}}}=
\{x_{\rho(k\!^{^V}_{\delta^{\star}})}\vert \rho(k\!^{^V}_{\delta^{\star}})\geq
\widetilde{\rho}(k\!^{^V}_{\delta^{\star}})\}$ with
$\widetilde{\rho}(k\!^{^V}_{\delta^{\star}})\geq
\rho_{_V}(k\!^{^V}_{\delta^{\star}})$
 \end{center}
\par\noindent
of a $\tau_{_p}$-net
$(x_{\rho(k\!^{^V}_{\delta^{\star}})})_
{\rho(k\!^{^V}_{\delta^{\star}}) \in
P_{(k\!^{^V}_{\delta^{\star}})}} \in {{\mathcal{A}}}_{{_\xi}_{_{\delta^{\star}}}}$ such that:
\par
\smallskip
\par\noindent
a$^{\prime\prime}$)\ The pair
$(I_{_{k\!\!^{^V}_{\delta^{\star}}}},{\hat{I}}_{_{k_{_{a_{_i}^{\star}}}}})$
is $V\circ V$-close for each $i<\beta$ and all $k_{_{a_{_i}^{\star}}}$.
\par
\bigskip
\par\noindent
b$^{\prime\prime}$)\ The pair
$(I_{_{k\!\!^{^V}_{\delta^{\star}}}},{I}_{{{\kappa}\!^{^{_V}}\
\!\!\!\!\!\!\!_{a\!^{\star}_{_i}}}})$ is $V\circ V\circ V$-close for
each $i<\beta$.
\par
\smallskip
\par
We put ${I}_{{{\kappa}\!^{^{_V}}\
\!\!\!\!\!\!\!_{a\!^{\star}_{_{\beta}}}}}=I_{_{k\!\!^{^V}_{\delta^{\star}}}}$
and $G^{^W}_{_{
a^{\star}_{_{\beta}}}}=(\displaystyle\bigcup_{i<\beta}G^{^W}_{_{
a^{\star}_{_i}}})\cup {I}_{{{\kappa}\!^{^{_V}}\
\!\!\!\!\!\!\!_{a\!^{\star}_{_{\beta}}}}}$.
\par

\par
\smallskip
\par
It is evident that the above properties ($A_1$) to ($B_1$) are
extended for each $i<\beta$.
\par
\smallskip
\par
(4) {\it The continuation of the process}.
\par\smallskip\par\noindent
We continue the process until the end of $(\xi)$, that is until the
\textquotedblleft exhausting\textquotedblright of the elements of
$(\xi)$ and we form the set $G^{^W}=\cup \{G^{^W}_{_\beta}\vert
\beta$ an ordinal$\}$. Thus, we have extracted from $(\xi)$ a subnet
$(\xi_{{a_{_{i}}}})_{i\in I_{_W}}=(\xi^{^W})$.

We consider the set
\par

\begin{center}
$\widetilde{A}^{^V}=\{\widetilde{a}^{^V}\vert\ \widetilde{a}^{^V}=
\rho({\kappa}\!^{^{_V}}\
\!\!\!\!\!_{a\!^{\star}_{_{i}}})
\in P({\kappa}\!^{^{_V}}\
\!\!\!\!\!_{a\!^{\star}_{_{i}}}), \  i\in I_{_W}\}$
\end{center}
to which we give the following order:
\par \ \
$\widetilde{a}^{^V}=\rho({\kappa}\!^{^{_V}}\
\!\!\!\!\!_{a\!^{\star}_{_{j}}})\leq
\widetilde{a^{\prime}}^{^V}=\rho({\kappa}\!^{^{_V}}\
\!\!\!\!\!_{a\!^{\star}_{_{i}}})$
\par\noindent
if and only if

\par
\ \ \ \ \ \ \ \ \ (i) \ \ $a^{\star}_{j}<a^{\star}_{i}$\ \ \ \ or
\par
\ \ \ \ \ \ \ \ (ii) \ ${\kappa}\!^{^{_V}}\
\!\!\!\!\!_{a\!^{\star}_{_j}}={\kappa}\!^{^{_V}}\
\!\!\!\!\!_{a\!^{\star}_{_i}}$ and $\rho({\kappa}\!^{^{_V}}\
\!\!\!\!\!_{a\!^{\star}_{_i}}) =\rho^{\prime}({\kappa}\!^{^{_V}}\
\!\!\!\!\!_{a\!^{\star}_{_j}})\geq \rho({\kappa}\!^{^{_V}}\
\!\!\!\!\!_{a\!^{\star}_{_j}})\geq
\widetilde{\rho}({\kappa}\!^{^{_V}}\
\!\!\!\!\!_{a\!^{\star}_{_j}})=\widetilde{\rho}({\kappa}\!^{^{_V}}\
\!\!\!\!\!_{a\!^{\star}_{_i}})$ (see fig. 14).

\par

\par

\pspicture(0,0)(0,0)

\pscircle[linewidth=2pt](1.43,-1.30){0.2}

\multiput(.2,-1.52)(1.46,-1.30){1}{*}

\multiput(.5,-1.52)(1.46,-1.30){1}{*}

\multiput(.8,-1.52)(1.46,-1.30){1}{*}

\Rput[ul](.1,-1){$(\xi)$}

\multiput(2.54,-1.52)(1.46,-1.30){1}{*}

\multiput(2.94,-1.52)(1.46,-1.30){1}{*}

\multiput(3.24,-1.52)(1.46,-1.30){1}{*}

\multiput(5.24,-1.52)(1.46,-1.30){1}{*}

\multiput(4.94,-1.52)(1.46,-1.30){1}{*}

\multiput(5.6,-1.52)(1.46,-1.30){1}{*}

\psline{->}(6,-1.30)(6.4,-1.30)

\Rput[ul](1,-.6){$\xi_{_{a^{\star}_{_j}}}$}

\Rput[ul](3.8,-.6){$\xi_{_{a^{\star}_{_i}}}$}

\psline[linestyle=dotted]{<->}(1.93,-4.4)(2.47,-1.51)

\Rput[ul](1.35,-4.9){$\rho({\kappa}\!^{^{_V}}\
\!\!\!\!\!_{a\!^{\star}_{_j}})$}

\psline{-}(.8,-4.2)(1.37,-1.40)

\Rput[ul](.51,-4.3){$\star$}

\Rput[ul](.29,-4.9){${I}_{{{\kappa}\!^{^{_V}}\
\!\!\!\!\!\!\!_{a\!^{\star}_{_j}}}}$}

\psline{-}(3.67,-4.25)(4.26,-1.44)

\Rput[ul](3.39,-4.35){$\star$}

\Rput[ul](3.21,-4.9){${I}_{{{\kappa}\!^{^{_V}}\
\!\!\!\!\!\!\!_{a\!^{\star}_{_i}}}}$}

\qdisk(2.03,-3.8){2pt}

\Rput[ul](2.19,-3.7){$
\widetilde{a}^{^V}$}

\multiput(1.8,-3.88)(1.1,-3.7){1}{-}

\multiput(1.6,-3.88)(1.1,-3.7){1}{-}

\multiput(1.4,-3.88)(1.1,-3.7){1}{-}

\psline[linestyle=dotted]{<-}(1.1,-3.8)(1.14,-3.8)

\Rput[ul](.6,-3.8){$\star$}

\Rput[ul](-.4,-3.5){$
x_{_{{\rho}({\kappa}\!^{^{_V}}\
\!\!\!\!\!\!\!_{a\!^{\star}_{_j}})}}$}

\Rput[ul](2.8,-2){$
x_{_{{\rho}({\kappa}\!^{^{_V}}\
\!\!\!\!\!\!\!_{a\!^{\star}_{_i}})}}$}

\Rput[ul](7.3,-2){$
x_{_{\rho}({\kappa}\!^{^{_V}}\
\!\!\!\!\!\!\!_{a\!^{\star}_{_i}})}$}

\qdisk(5.22,-1.9){2pt}

\Rput[ul](5.43,-1.9){$
\widetilde{a^{\prime}}^{^V}$}

\multiput(5.02,-2)(1.1,-3.7){1}{-}

\multiput(4.82,-2)(1.1,-3.7){1}{-}

\multiput(4.62,-2)(1.1,-3.7){1}{-}

\psline[linestyle=dotted]{<-}(4.36,-1.91)(4.6,-1.91)

\Rput[ul](3.86,-1.91){$\star$}

\psline[linestyle=dotted]{<->}(4.7,-4.4)(5.3,-1.5)

\Rput[ul](4.27,-4.8){$\rho({\kappa}\!^{^{_V}}\
\!\!\!\!\!_{a\!^{\star}_{_i}})$}

\pscircle[linewidth=2pt](4.23,-1.30){0.2}

\pscircle[linewidth=2pt](8.93,-1.30){0.2}

\multiput(7.6,-1.52)(1.46,-1.30){1}{*}

\multiput(7.9,-1.52)(1.46,-1.30){1}{*}

\multiput(8.2,-1.52)(1.46,-1.30){1}{*}

\Rput[ul](7.6,-1){$(\xi)$}

\Rput[ul](8.5,-.6){$\xi_{_{a^{\star}_{_j}}}=\xi_{_{a^{\star}_{_i}}}$}

\multiput(10.1,-1.52)(1.46,-1.30){1}{*}

\multiput(9.8,-1.52)(1.46,-1.30){1}{*}

\multiput(9.5,-1.52)(1.46,-1.30){1}{*}

\psline{->}(10.46,-1.30)(10.9,-1.30)

\psline{-}(8.3,-4.2)(8.87,-1.40)

\Rput[ul](8.01,-4.3){$\star$}

\Rput[ul](6.79,-4.9){${I}_{{{\kappa}\!^{^{_V}}\
\!\!\!\!\!\!\!_{a\!^{\star}_{_j}}}}=
{I}_{{{\kappa}\!^{^{_V}}\
\!\!\!\!\!\!\!_{a\!^{\star}_{_i}}}}
$}

\Rput[ul](8.79,-4.9){$
\rho({\kappa}\!^{^{_V}}\
\!\!\!\!\!_{a\!^{\star}_{_j}})=\rho({\kappa}\!^{^{_V}}\
\!\!\!\!\!_{a\!^{\star}_{_i}})$}

\Rput[ul](6.95,-3.6){$
x_{_{\rho}({\kappa}\!^{^{_V}}\
\!\!\!\!\!\!\!_{a\!^{\star}_{_j}})}$}

\Rput[ul](8.51,-1.91){$\star$}

\Rput[ul](8.1,-3.8){$\star$}

\psline[linestyle=dotted]{<->}(9.2,-4.4)(9.8,-1.5)

\Rput[ul](4.27,-4.8){$\rho({\kappa}\!^{^{_V}}\
\!\!\!\!\!_{a\!^{\star}_{_i}})$}

\qdisk(9.72,-1.9){2pt}

\Rput[ul](9.83,-1.9){$
\widetilde{a^{\prime}}^{^V}$}

\multiput(9.47,-2)(1.1,-3.7){1}{-}

\multiput(9.27,-2)(1.1,-3.7){1}{-}

\psline[linestyle=dotted]{<-}(8.9,-1.91)(9.16,-1.91)

\qdisk(9.34,-3.8){2pt}

\Rput[ul](9.43,-3.8){$
\widetilde{a}^{^V}$}

\multiput(9.1,-3.9)(1.1,-3.7){1}{-}

\multiput(8.9,-3.9)(1.1,-3.7){1}{-}

\psline[linestyle=dotted]{<-}(8.6,-3.8)(8.7,-3.8)

\endpspicture

\par\bigskip\bigskip\bigskip\bigskip\bigskip\bigskip\bigskip\bigskip\bigskip\bigskip\bigskip\bigskip\bigskip
\bigskip\bigskip\bigskip\bigskip\bigskip
\begin{center}
Figure 14
\end{center}
\par\smallskip\par

\par
\smallskip
Thus, we have constructed the net $G^{^W}=(x_{{\rho}({\kappa}\!^{^{_V}}\
\!\!\!\!\!\!\!_{a\!^{\star}_{_i}})})_{_{ \rho({\kappa}\!^{^{_V}}\
\!\!\!\!\!\!\!_{a\!^{\star}_{_i}})\in \widetilde{A}^{^V}}}$. It is easy
to show that the validity of the Property $(A_1)$ for each ordinal $\beta$ imply that
$(x_{{\rho}({\kappa}\!^{^{_V}}\
\!\!\!\!\!\!\!_{a\!^{\star}_{_i}})})_{_{ \rho({\kappa}\!^{^{_V}}\
\!\!\!\!\!\!\!_{a\!^{\star}_{_i}})\in \widetilde{A}}}$ is a
$\tau_{_p}$-net for $W$
as well as
$(\phi(x_{{\rho}({\kappa}\!^{^{_V}}\
\!\!\!\!\!\!\!_{a\!^{\star}_{_i}})}))_{_{ \rho({\kappa}\!^{^{_V}}\
\!\!\!\!\!\!\!_{a\!^{\star}_{_i}})\in \widetilde{A}^{^V}}}$ and
$({\xi}_{a^\star})_{a^\star\in {A}^\star}$ are left cofinal
for $\overline{W}$ (see fig. 15).
\end{proof}

\par

\bigskip\bigskip\bigskip
\par

\par

\pspicture(0,0)(0,0)

\pscircle[linewidth=2pt](1.43,-1.30){0.2}

\pscircle[linewidth=2pt](3.33,-1.30){0.2}

\pscircle[linewidth=2pt](5.23,-1.30){0.2}

\multiput(1.94,-1.52)(1.46,-1.30){1}{*}

\multiput(2.24,-1.52)(1.46,-1.30){1}{*}

\multiput(2.54,-1.52)(1.46,-1.30){1}{*}

\multiput(4.64,-1.52)(1.46,-1.30){1}{*}

\multiput(4.34,-1.52)(1.46,-1.30){1}{*}

\multiput(4.04,-1.52)(1.46,-1.30){1}{*}

\multiput(6.36,-1.52)(1.46,-1.30){1}{*}

\multiput(6.08,-1.52)(1.46,-1.30){1}{*}

\multiput(5.80,-1.52)(1.46,-1.30){1}{*}

\psline{->}(6.8,-1.30)(7.4,-1.30)

\multiput(.2,-1.52)(1.46,-1.30){1}{*}

\multiput(.5,-1.52)(1.46,-1.30){1}{*}

\multiput(.8,-1.52)(1.46,-1.30){1}{*}

\Rput[ul](.1,-1){$(\xi^{^W})$}

\Rput[ul](1,-.6){$\xi_{_{a^{\star}_{_i}}}$}

\Rput[ul](3,-.6){$\xi_{_{a^{\star}_{_{i^{\prime}}}}}$}

\Rput[ul](5,-.6){$\xi_{_{a^{\star}_{_{i^{\prime\prime}}}}}$}

\Rput[ul](7,-.6){$(i\in I_{_W})$}

\psline{-}(1,-2.8)(1.37,-1.40)

\Rput[ul](.74,-2.8){$\star$}

\Rput[ul](.71,-3.4){${I}_{{{\kappa}\!^{^{_V}}\
\!\!\!\!\!\!\!_{a\!^{\star}_{_{i}}}}}$}

\psline{-}(3,-2.8)(3.37,-1.40)

\Rput[ul](2.74,-2.8){$\star$}

\Rput[ul](2.7,-3.4){${I}_{{{\kappa}\!^{^{_V}}\
\!\!\!\!\!\!\!_{a\!^{\star}_{_{i^{\prime}}}}}}$}

\psline{-}(4.8,-2.8)(5.17,-1.40)

\Rput[ul](4.54,-2.8){$\star$}

\Rput[ul](4.49,-3.4){${I}_{{{\kappa}\!^{^{_V}}\
\!\!\!\!\!\!\!_{a\!^{\star}_{_{i^{\prime\prime}}}}}}$}

\psline[linestyle=dotted]{->}(.1,-2.8)(9.4,-2.8)

\psline[linestyle=dotted]{->}(.1,-1.5)(9.4,-1.5)

\endpspicture

\par\bigskip\bigskip\bigskip\bigskip\bigskip\bigskip\bigskip\bigskip
\bigskip\bigskip\bigskip\bigskip\bigskip
\begin{center}
Figure 15
\end{center}
\par\bigskip\smallskip\par

\par
\smallskip
A similar demonstration gives the following:
\par

\begin{lemma}\label{a22}{\rm Let $({\eta}_{_{\beta^\star}})_{_{\beta^\star\in {B}^\star}}$ be a $\tau_{_q}$-conet
in $(\overline{X},\overline{\mathcal{U}})$ without end point such that
for each $\delta^{\star}>\beta^{\star}$, ${{\mathcal{A}}}_{_{\eta_{_{\delta^{\star}}}}}\nsubseteq {{\mathcal{A}}}_{_{\eta_{_{\beta^{\star}}}}}$.
Let also $W\in \mathcal{U}$ be an entourage in $\mathcal{U}_{_0}$.
Then, there is a $\tau_{_q}$-conet for
$W$, $(s_{_\mu})_{_{\mu\in M}}$, such that the nets
$({\eta}_ {\beta^\star})_{\beta^\star\in {B}^\star}$ and
$(\phi(s_{_\mu})){_{_{\mu\in M}}}$ are right cofinal for $W$.}
\end{lemma}

\begin{lemma}\label{a23}{\rm Let $({\xi}_ {a^\star})_{a^\star\in
{A}^\star}$ be a $\tau_{_p}$-net in $(\overline{X},\overline{\mathcal{U}})$
without end point such that for each $\gamma^{\star}>a^{\star}$,
${{\mathcal{A}}}_{_{\xi_{_{a^{\star}}}}}\nsubseteq {{\mathcal{A}}}_{_{\xi_{_{\gamma^{\star}}}}}$. Then, there is a $\tau_{_p}$-net
$(t_{_\lambda})_{_{\lambda\in \Lambda}}$ of $(X,{\mathcal{U}})$ such that
the nets $({\xi}_{a^\star})_{a^\star\in {A}^\star}$ and
$(\phi(t_{_\lambda})){_{_{\lambda\in \Lambda}}}$ are left cofinal.}
\end{lemma}

\begin{proof} Let $(\xi)=({\xi}_{a^\star})_{a^\star\in {A}^\star}$ be as above. By the construction
of Lemma \ref{a21}, for each $W\in {\mathcal{U}}$, we consider: (1) The subnet
$(\xi^{^W})=(\xi_{{a^{\star}_{_{i}}}})_{i\in
I_{_W}}=(\xi_{{a^{\star}_{_{i(W)}}}})_{i(W)\in I_{_W}}$ of
$(\xi)$; (2) The net $G^{^W}=(x_{{\rho}({\kappa}\!^{^{_V}}\
\!\!\!\!\!\!\!_{a\!^{\star}_{_{i(W)}}})})_{_{
\rho({\kappa}\!^{^{_V}}\ \!\!\!\!\!\!\!_{a\!^{\star}_{_{i(W)}}})\in
\widetilde{A}^{^V}}}$ of $X$ which corresponds to a $W\in {\mathcal{U}}_{_0}$
($V\in {\mathcal{U}}_{_0}$, $V\circ V\circ V\subseteq W$);\footnote{We
remind that the index $V$ in the form ${\kappa}\!^{^{_V}}\
\!\!\!\!\!\!\!_{a\!^{\star}_{_{i(W)}}}$ refers to the different
$\tau_{_p}$-nets which have been chosen in the concrete points
$\xi_{_{a^{\star}}}$ in the process of Lemma \ref{a21} with $W$
(and hence $V$) changeable.
If $W^{\prime}\neq W$, then
$a^{\star}_{_{i(W^{\prime})}}
=a^{\star}_{_{i(W)}}=a^{\star}$ does not imply
${\kappa}\!^{^{_{V\!^{^\prime}}}}\
\!\!\!\!\!\!\!_{a\!^{\star}_{_{i(W\!^{^\prime})}}}={\kappa}\!^{^{_{V}}}\
\!\!\!\!\!\!\!_{a\!^{\star}_{_{i(W)}}}$
in general.} (3) The
final segments ${I}_{{{\kappa}\!^{^{_V}}\
\!\!\!\!\!\!\!_{a\!^{\star}_{_{i(W)}}}}}$ as well as their extreme
points $x_{\widetilde{\rho}({\kappa}\!^{^{_V}}\
\!\!\!\!\!\!\!_{a\!^{\star}_{_{i(W)}}})}$
($\widetilde{\rho}({\kappa}\!^{^{_V}}\
\!\!\!\!\!_{a\!^{\star}_{_{i(W)}}})\geq
\rho_{_V}({\kappa}\!^{^{_V}}\ \!\!\!\!\!_{a\!^{\star}_{_{i(W)}}})$).

\par
\bigskip
\par
\par
(1) {\it The construction of the desired $\tau_{_p}$-net $G$.}
\par\noindent

\par
\smallskip
\par\noindent
We put $G=\bigcup \{ G^{^W}\ \vert W \in {\mathcal{U}} \}.$
\par\noindent

\par
We consider the set
\par
\begin{center}
$\overline{A}=\{\overline{a}=\rho({\kappa}\!^{^{_{V}}}\
\!\!\!\!\!_{a\!^{\star}_{_{i(W)}}})\vert
\rho({\kappa}\!^{^{_{V}}}\
\!\!\!\!\!_{a\!^{\star}_{_{i(W)}}})\geq \widetilde{\rho}({\kappa}\!^{^{_{V}}}\
\!\!\!\!\!_{a\!^{\star}_{_{i(W)}}}), W\in \mathcal{U}\}$
\end{center}
\par\noindent
to which we give the following order:
\begin{center}
$\overline{a}=\rho_{_0}({\kappa}\!^{^{_{V_{_\sigma}}}}\
\!\!\!\!\!_{a\!^{\star}_{_{i(W_{_\sigma})}}})\leq \overline{a^{\prime}}=
\rho_{_1}({\kappa}\!^{^{_{V_{_\pi}}}}\
\!\!\!\!\!_{a\!^{\star}_{_{i(W_{_\pi})}}})$
\end{center}
\smallskip
\par\noindent
if and only if
\par\smallskip
\par\noindent
(i) $a^{\star}_{i(W_{_\sigma})}<a^{\star}_{i(W_{_\pi})}$ and there exists a point
$\xi_{_{\gamma^{\star}}}=\xi_{_{a^{\star}_{_{i_{_{\gamma^{\star}}}(W_{_\pi})}}}}\in
(\xi_{_{a^{\star}_{_{i(W_{_\pi})}}}})_{i(W_{_\pi})\in I_{_{W_\pi}}}$
such that $a^{\star}_{i(W_{_\sigma})}\leq
a^{\star}_{i_{_{\gamma^{\star}}}(W_{_\pi})}<a^{\star}_{i(W_{_\pi})}$,
(it is assumed that if $W_\pi\subseteq W_\sigma$, then
$V_\pi\subseteq V_\sigma$ as well) (see fig. 16)

\par

or
\par
\smallskip
\par\noindent
(ii) $a^{\star}_{i(W_{_\pi})}=a^{\star}_{i(W_{_\sigma})}$,
${\kappa}\!^{^{_{V_{_\pi}}}}\ \!\!\!\!\!\!\!_{a\!^{\star}_{_{i(W_{_\pi})}}}=
{\kappa}\!^{^{_{V_{_\sigma}}}}\ \!\!\!\!\!\!\!_{a\!^{\star}_{_{i(W_{_\sigma})}}}$ and
$\overline{a^{\prime}}=
\rho_{_1}({\kappa}\!^{^{_{V_{_\pi}}}}\ \!\!\!\!\!\!\!_{a\!^{\star}_{_{i(W_{_\pi})}}})=
\rho_{_2}({\kappa}\!^{^{_{V_{_\sigma}}}}\ \!\!\!\!\!\!\!\!_{a\!^{\star}_{_{i(W_{_\sigma})}}})\geq
\rho_{_0}({\kappa}\!^{^{_{V_{_\sigma}}}}\
\!\!\!\!\!\!\!\!_{a\!^{\star}_{_{i(W_{_\sigma})}}})$ (see fig. 17).

It is clear that $\overline{A}$ is nonempty. The relation $\leq$ is
a right filtering preorder and hence the index-set $\overline{A}$ is
directed. We only prove the transitivity of $\leq$, the rest are
trivial. Indeed, let
$\overline{c}=\rho^{\ast}({\kappa}\!^{^{_{V_{_\theta}}}}\
\!\!\!\!\!\!\!_{a\!^{\star}_{_{i(W_{_\theta})}}})$ be
such that $\overline{a}\leq \overline{a^{\prime}}$ and
$\overline{a^{\prime}}\leq \overline{c}$. The only interesting case
is if
$a^{\star}_{i(W_{_\sigma})}<a^{\star}_{i(W_{_\pi})}<a^{\star}_{i(W_{_\theta})}$
with $V_\theta\subseteq V_\pi\subseteq V_\sigma$. Then, there are
$\xi^{\star}_{\gamma^{\star}}\in
(\xi_{_{a^{\star}_{_{i(W_{_\pi})}}}})_{i(W_{_\pi})\in I_{_{W_\pi}}}$
 and $\xi^{\star}_{\delta^{\star}}\in
(\xi_{_{a^{\star}_{_{i(W_{_\theta})}}}})_{i(W_{_\theta})\in
I_{_{W_\theta}}}$
 such that
$a^{\star}_{i(W_{_\sigma})}\leq
\gamma^{\star}<a^{\star}_{i(W_{_\pi})}\leq
\delta^{\star}<a^{\star}_{i(W_{_\theta})}$. Because of the existence
of that $\delta^{\star}$, we have $\overline{a^{\prime}}\leq
\overline{c}$.

\par

\vfill\eject

\par

\pspicture(0,0)(0,0)

\pscircle[linewidth=2pt](1.43,-1.30){0.2}

\pscircle[linewidth=2pt](3.33,-1.30){0.2}

\pscircle[linewidth=2pt](5.23,-1.30){0.2}

\multiput(1.94,-1.52)(1.46,-1.30){1}{*}

\multiput(2.24,-1.52)(1.46,-1.30){1}{*}

\multiput(2.54,-1.52)(1.46,-1.30){1}{*}

\multiput(4.64,-1.52)(1.46,-1.30){1}{*}

\multiput(4.34,-1.52)(1.46,-1.30){1}{*}

\multiput(4.04,-1.52)(1.46,-1.30){1}{*}

\multiput(6.36,-1.52)(1.46,-1.30){1}{*}

\multiput(6.08,-1.52)(1.46,-1.30){1}{*}

\multiput(5.80,-1.52)(1.46,-1.30){1}{*}

\psline{->}(6.8,-1.30)(7.4,-1.30)

\multiput(.2,-1.52)(1.46,-1.30){1}{*}

\multiput(.5,-1.52)(1.46,-1.30){1}{*}

\multiput(.8,-1.52)(1.46,-1.30){1}{*}

\Rput[ul](.1,-1){$(\xi^{^{W_\pi}})$}

\Rput[ul](5,-.6){$\xi_{_{a^{\star}_{_{i(W_\pi)}}}}$}

\Rput[ul](3,-.6){$\xi_{_{a^{\star}_{_{i_{_{\gamma^{\star}}}(W_\pi)}}}}$}

\Rput[ul](1,-.6){$\xi_{_{a^{\star}_{_{i^{\prime}(W_\pi)}}}}$}

\Rput[ul](5.8,-2.5){$
\rho({\kappa}\!^{^{_{V_\pi}}}\
\!\!\!\!\!_{a\!^{\star}_{_{i(W_\pi)}}})$}

\psline[linestyle=dotted]{-}(5.1,-1.8)(5.15,-1.5)

\psline[linestyle=dotted]{-}(5.3,-1.9)(5.8,-1.85)

\psline{<-}(5.2,-1.91)(5.3,-1.91)

\Rput[ul](4.8,-1.9){$\star$}

\qdisk(5.1,-2.1){1pt}

\qdisk(5.15,-2.2){1pt}

\qdisk(5,-2.3){1pt}

\qdisk(5.07,-2.4){1pt}

\Rput[ul](4.8,-2.6){$\star$}

\qdisk(5.87,-1.8){2pt}

\Rput[ul](5.9,-1.78){$\overline{a^{\prime}}$}

\multiput(5.7,-3)(1.1,-3.7){1}{-}

\multiput(5.9,-3)(1.1,-3.7){1}{-}

\multiput(6.1,-3)(1.1,-3.7){1}{-}

\psline{->}(6.3,-2.91)(6.5,-2.91)

\multiput(4.3,-3)(1.1,-3.7){1}{-}

\multiput(4.1,-3)(1.1,-3.7){1}{-}

\multiput(3.9,-3)(1.1,-3.7){1}{-}

\psline{<-}(3.55,-2.91)(3.75,-2.91)

\Rput[ul](4.5,-3.1){$
{\kappa}\!^{^{_{V_{_\pi}}}}\
\!\!\!\!\!_{a\!^{\star}_{_{i(W_{_\pi})}}}$}

\Rput[ul](8,-4.6){($i({W_{\sigma}})\in I_{_{W_\sigma}})$}

\Rput[ul](7,-.6){($i({W_{\pi}})\in I_{_{W_\pi}})$}

\pscircle[linewidth=2pt](2.13,-5.30){0.2}

\pscircle[linewidth=2pt](4.83,-5.30){0.2}

\pscircle[linewidth=2pt](6.93,-5.30){0.2}

\Rput[ul](.6,-4.8){$(\xi^{^{W_\sigma}})$}

\Rput[ul](3,-4.78){$\xi_{_{a^{\star}_{_{i_{_{\gamma^{\star}}}(W_\pi)}}}}$}

\multiput(1.44,-5.52)(1.46,-1.30){1}{*}

\multiput(1.1,-5.52)(1.46,-1.30){1}{*}

\multiput(.74,-5.52)(1.46,-1.30){1}{*}

\psline[linestyle=dotted]{-}(2.05,-5.5)(2,-6)

\psline[linestyle=dotted]{-}(3.3,-1.4)(3.3,-4.38)

\multiput(3.64,-5.52)(1.46,-1.30){1}{*}

\multiput(3.34,-5.52)(1.46,-1.30){1}{*}

\multiput(3.04,-5.52)(1.46,-1.30){1}{*}

\multiput(6.06,-5.52)(1.46,-1.30){1}{*}

\multiput(5.78,-5.52)(1.46,-1.30){1}{*}

\multiput(5.50,-5.52)(1.46,-1.30){1}{*}

\psline{->}(6.8,-1.30)(7.4,-1.30)

\multiput(7.4,-5.52)(1.46,-1.30){1}{*}

\multiput(7.7,-5.52)(1.46,-1.30){1}{*}

\multiput(8,-5.52)(1.46,-1.30){1}{*}

\psline{->}(8.45,-5.32)(9.1,-5.32)

\Rput[ul](.1,-1){$(\xi^{^{W_\pi}})$}

\Rput[ul](1.8,-4.7){$\xi_{_{a^{\star}_{_{i(W_\sigma)}}}}$}

\Rput[ul](4.5,-4.7){$\xi_{_{a^{\star}_{_{i^{\prime}(W_\sigma)}}}}$}

\Rput[ul](6.5,-4.7){$\xi_{_{a^{\star}_{_{i^{\prime\prime}(W_\sigma)}}}}$}

\Rput[ul](3,-6.5){$
\rho({\kappa}\!^{^{_{V_\sigma}}}\
\!\!\!\!\!_{a\!^{\star}_{_{i(W_\sigma)}}})$}

\qdisk(2.1,-6.1){1pt}

\qdisk(2.15,-6.2){1pt}

\qdisk(2,-6.3){1pt}

\qdisk(2.07,-6.4){1pt}

\Rput[ul](1.8,-6.6){$\star$}

\psline[linestyle=dotted]{<->}(5.8,-2.7)(5.9,-1.5)

\psline[linestyle=dotted]{<->}(3,-6.7)(3.15,-5.5)

\psline{<-}(3.55,-2.91)(3.75,-2.91)

\Rput[ul](4.5,-3.1){$
{\kappa}\!^{^{_{V_{_\pi}}}}\
\!\!\!\!\!_{a\!^{\star}_{_{i(W_{_\pi})}}}$}

\multiput(5.7,-3)(1.1,-3.7){1}{-}

\multiput(5.9,-3)(1.1,-3.7){1}{-}

\multiput(6.1,-3)(1.1,-3.7){1}{-}

\psline{->}(6.3,-2.91)(6.5,-2.91)

\multiput(4.3,-3)(1.1,-3.7){1}{-}

\multiput(4.1,-3)(1.1,-3.7){1}{-}

\multiput(3.9,-3)(1.1,-3.7){1}{-}

\Rput[ul](4.4,-1.8){$x_{_{\overline{a^{\prime}}}}$}

\Rput[ul](1.3,-6){$x_{_{\overline{a}}}$}

\Rput[ul](1.7,-6){$\star$}

\qdisk(3.1,-6){2pt}

\psline[linestyle=dotted]{<-}(2.1,-5.98)(3,-6)

\Rput[ul](3.1,-5.98){$\overline{a}$}

\multiput(2.7,-7.2)(1.1,-3.7){1}{-}

\multiput(2.9,-7.2)(1.1,-3.7){1}{-}

\multiput(3.1,-7.2)(1.1,-3.7){1}{-}

\psline{->}(3.3,-7.11)(3.5,-7.11)

\multiput(1.3,-7.2)(1.1,-3.7){1}{-}

\multiput(1.1,-7.2)(1.1,-3.7){1}{-}

\multiput(.9,-7.2)(1.1,-3.7){1}{-}

\psline{<-}(.55,-7.11)(.75,-7.11)

\Rput[ul](1.5,-7.3){$
{\kappa}\!^{^{_{V_{_\sigma}}}}\
\!\!\!\!\!_{a\!^{\star}_{_{i(W_{_\sigma})}}}$}

\endpspicture

\par

\bigskip\bigskip\bigskip\bigskip\bigskip\bigskip\bigskip
\bigskip\bigskip\bigskip\bigskip\bigskip\bigskip\bigskip\bigskip\bigskip\bigskip\bigskip\bigskip
\begin{center}
Figure 16
\end{center}

\bigskip
\par
(2){\it The proof that $G$ is a $\tau_{_p}$-net.}
\par\noindent

\par
Let $W\in {\mathcal{U}}_{_0}$. Suppose that $V=W_1\in {\mathcal{U}}_{_0}$ is
the corresponded entourage to $W$ from the lemma \ref{a21}. Similarly we
consider the entourage $V_1$ which corresponds to $W_1$. There holds
$V\circ V\circ V\subseteq W $ and $V_1\circ V_1\circ V_1\subseteq
W_1 $. Let $a^{\star}_{_{\overline{V}_{_1}}}$ be the extreme index
of $(\xi)$ for $\overline{V}_{_1}$. Let $a^{\star}_{_\lambda}\geq
{a^{\star}}_{_{\overline{V}_{_1}}}$ with
$\xi_{_{a^{\star}_{_\lambda}}}=\xi_{_{a^{\star}_{_{i_{_0}(W_{_1})}}}}\in
(\xi_{_{a^{\star}_{_{i(W_{_1})}}}})_{_{i(W_{_1})\in I_{_{W_{_1}}}}}$
and let
$\overline{a^{\prime}}=\rho_{_1}({\kappa}\!^{^{_{V_{_\pi}}}}\
\!\!\!\!\!_{a\!^{\star}_{_{i(W_{_\pi})}}})\geq
\overline{a}=\rho_{_0}({\kappa}\!^{^{_{V_{_\sigma}}}}\
\!\!\!\!\!_{a\!^{\star}_{_{i(W_{_\sigma})}}})\geq \overline{a}_{_0}=
\widetilde{\rho}({\kappa}\!^{^{_{V_{_1}}}}\
\!\!\!\!\!\!\!_{a\!^{\star}_{_{i_{_0}(W_{_1})}}})$ (there
holds $a^{\star}_{_{\overline{V}}}\leq
a^{\star}_{i_{_0}(W_{_1})}\leq a^{\star}_{i(W_{_\sigma})}\leq
a^{\star}_{i(W_{_\pi})}$ and $W_\pi\subseteq W_\sigma \subseteq
W_{_1}\subseteq W$ ). We distinguish two cases: (i)
$a^{\star}_{i(W_{_\pi})}=a^{\star}_{i(W_{_\sigma})}$,
${\kappa}\!^{^{_{V_{_\pi}}}}\ \!\!\!\!\!\!\!_{a\!^{\star}_{_{i(W_{_\pi})}}}=
{\kappa}\!^{^{_{V_{_\sigma}}}}\ \!\!\!\!\!\!\!_{a\!^{\star}_{_{i(W_{_\sigma})}}}$ and
$\rho_{_1}({\kappa}\!^{^{_{V_{_\pi}}}}\ \!\!\!\!\!\!\!_{a\!^{\star}_{_{i(W_{_\pi})}}})=
\rho_{_2}({\kappa}\!^{^{_{V_{_\sigma}}}}\ \!\!\!\!\!\!\!\!_{a\!^{\star}_{_{i(W_{_\sigma})}}})\geq
\rho_{_0}({\kappa}\!^{^{_{V_{_\sigma}}}}\
\!\!\!\!\!\!\!\!_{a\!^{\star}_{_{i(W_{_\sigma})}}})
\geq
\widetilde{\rho}({\kappa}\!^{^{_{V_{_\sigma}}}}\
\!\!\!\!\!\!\!_{a\!^{\star}_{_{i(W_{_\sigma})}}})$. Then

\begin{center}
$(x_{{\!\!\rho_{_1}}({\kappa}\!^{^{_{V_{_\pi}}}}\
\!\!\!\!\!\!\!\!\!_{a\!^{\star}_{_{i(W_{_\pi})}}})},x_{{\!\!\rho_{_0}}({\kappa}\!^{^{_{V_{_\sigma}}}}\
\!\!\!\!\!\!\!\!\!_{a\!^{\star}_{_{i(W_{_\sigma})}}})})\in
V_\sigma\subseteq W$.
\end{center}
\par
\smallskip
\par\noindent
(ii) $a^{\star}_{i(W_{_\sigma})}<a^{\star}_{i(W_{_\pi})}$, \
$V_\pi\subseteq V_\sigma$ and there exists a point
$\xi^{\star}_{\gamma^{\star}}=\xi_{_{a^{\star}_{_{i_{_{\gamma^{\star}}}(W_{_\pi})}}}}\in
(\xi_{_{a^{\star}_{_{i(W_{_\pi})}}}})_{i(W_{_\pi})\in I_{_{W_\pi}}}$
such that $a^{\star}_{i(W_{_\sigma})}\leq
a^{\star}_{i_{_{\gamma^{\star}}}(W_{_\pi})}<a^{\star}_{i(W_{_\pi})}$.
Then, from the extended ($B_1$) a)-property (for each ordinal
$\beta$) in the Lemma \ref{a21}, we conclude that

\begin{center}
$(x_{{\rho}({\kappa}\!^{^{_{V_{_\pi}}}}\
\!\!\!\!\!\!\!\!\!_{a\!^{\star}_{_{i(W_{_\pi})}}})},t)\in V_\pi\circ
V_\pi$\ \ \ \ (1)
\end{center}
\par
\smallskip
\par\noindent
for each $t\in {\hat{I}}_{_{k_{_{a_{_{
i_{_{\gamma^{\star}}}(W_\pi)}}^{\star}}}}}$
(${\hat{I}}_{_{k_{_{a_{_{
i_{_{\gamma^{\star}}}(W_\pi)}}^{\star}}}}}$ are final segments of
the $k_{_{a_{_{ i_{_{\gamma^{\star}}}(W_\pi)}}^{\star}}}$
$\tau_{_p}$-nets of $\xi_{_{a_{_{
i_{_{\gamma^{\star}}}(W_\pi)}}^{\star}}}$).
Since
$(\xi_{_{a^{\star}_{_{i_{_{\gamma^{\star}}}(W_{_\pi})}}}},\xi_{_{a^{\star}_{i(W_{_\sigma})}}})\in
\overline{W}_{_\sigma}$ there is a final segment $I$ of a $\tau_{_p}$-net
of $\xi_{_{a_{_{ i_{_{\gamma^{\star}}}(W_\pi)}}^{\star}}}$,\
$I\subset {\hat{I}}_{_{k_{_{a_{_{
i_{_{\gamma^{\star}}}(W_\pi)}}^{\star}}}}}$ for some $k_{_{a_{_{
i_{_{\gamma^{\star}}}(W_\pi)}}^{\star}}}\in I_{_{W_{_\pi}}}$, for
which there are final segments
$I^{\star}_{_{k_{_{a^{\star}_{i(W_{_\sigma})}}}}}$ of all
$\tau_{_p}$-nets of $\xi_{a^{\star}_{i(W_{_\sigma})}}$ such that
the pair $(I,I^{\star}_{_{k_{_{a^{\star}_{i(W_{_\sigma})}}}}})$ is
finally $V_{_\sigma}$-close. Hence,
\begin{center}
$(I,I^{^\star}_{{{\kappa}\!^{^{{V_{_\sigma}}}}\
\!\!\!\!\!\!\!\!\!\!_{_{a^{\star}_{i(W_{_\sigma})}}}}})$ is finally
$V_{_\sigma}$-close\ \ \ \ (2).
\end{center}
\par
\smallskip
\par\noindent
Finally, we have \par
\bigskip
\par\ \ \ \ \ \ \ \
$I^{^\star}_{{{\kappa}\!^{^{{V_{_\sigma}}}}\
\!\!\!\!\!\!\!\!\!\!_{_{a^{\star}_{i(W_{_\sigma})}}}}}\subseteq
I_{{{\kappa}\!^{^{{V_{_\sigma}}}}\
\!\!\!\!\!\!\!\!\!\!_{_{a^{\star}_{i(W_{_\sigma})}}}}}$ or
$I_{{{\kappa}\!^{^{{V_{_\sigma}}}}\
\!\!\!\!\!\!\!\!\!\!_{_{a^{\star}_{i(W_{_\sigma})}}}}}\subseteq
I^{^\star}_{{{\kappa}\!^{^{{V_{_\sigma}}}}\
\!\!\!\!\!\!\!\!\!\!_{_{a^{\star}_{i(W_{_\sigma})}}}}}$\ \ \ \ \ \
(3).
\par
\bigskip
\par\noindent
From (1), (2) and (3) we conclude that

\begin{center}
$(x_{{\rho}({\kappa}\!^{^{_{V_{_\pi}}}}\
\!\!\!\!\!\!\!\!\!_{a\!^{\star}_{_{i(W_{_\pi})}}})},x_{{\rho}({\kappa}\!^{^{_{V_{_\sigma}}}}\
\!\!\!\!\!\!\!\!\!_{a\!^{\star}_{_{i(W_{_\sigma})}}})})\in
V_\pi\circ V_\pi\circ V_\sigma\circ V_\sigma\subseteq W$.
\end{center}

The index
$\overline{a}_{_0}=\widetilde{\rho}({\kappa}\!^{^{_{V_{_1}}}}\
\!\!\!\!\!\!\!_{a\!^{\star}_{_{i_{_0}(W_{_1})}}})$ is the
extreme index of $G$ for $W$. Thus $G$ is a $\tau_{_p}$-net in $X$.
\end{proof}

\smallskip
\par\noindent

\pspicture(0,0)(0,0)

\pscircle[linewidth=2pt](5.23,-1.30){0.2}

\multiput(4.24,-1.52)(1.46,-1.30){1}{*}
\multiput(3.94,-1.52)(1.46,-1.30){1}{*}
\multiput(3.54,-1.52)(1.46,-1.30){1}{*}

\multiput(6.56,-1.52)(1.46,-1.30){1}{*}
\multiput(6.28,-1.52)(1.46,-1.30){1}{*}
\multiput(6,-1.52)(1.46,-1.30){1}{*}

\psline{->}(7,-1.30)(7.6,-1.30)

\Rput[ul](2.3,-1){$(\xi^{^{W_\pi}})\cap (\xi^{^{W_\sigma}})$}
\Rput[ul](5,-.6){$\xi_{_{a^{\star}_{_{i(W_\pi)}}}}=\xi_{_{a^{\star}_{_{i(W_\sigma)}}}}$}

\psline[linestyle=dotted]{-}(5.2,-1.6)(5.2,-2.2)

\qdisk(5.1,-2.35){1pt}

\Rput[ul](4.8,-2.55){$\star$}

\Rput[ul](4.3,-2.55){$x_{_{\overline{a}}}$}

\qdisk(5.2,-2.70){1pt}

\qdisk(5.05,-2.85){1pt}

\qdisk(5.18,-3){1pt}

\qdisk(5.1,-3.15){1pt}

\Rput[ul](5.02,-3.3){$\star$}

\Rput[ul](4.4,-3.25){$x_{_{\overline{a^{\prime}}}}$}

\qdisk(5.1,-3.45){1pt}

\qdisk(5.15,-3.6){1pt}

\Rput[ul](5.12,-3.75){$\star$}

\psline[linestyle=dotted]{<->}(6.63,-1.6)(6.6,-3.7)

\qdisk(6.62,-2.2){2pt}

\Rput[ul](6.7,-2.15){$\overline{a}$}

\Rput[ul](6.6,-2.75){$\overline{a^{\prime}}$}

\psline[linestyle=dotted]{->}(6.61,-2.2)(5.25,-2.5)

\qdisk(6.62,-2.83){2pt}

\psline[linestyle=dotted]{->}(6.61,-2.83)(5.42,-3.3)

\Rput[ul](6.1,-4.2){$
{\rho}({\kappa}\!^{^{_{V_{_\sigma}}}}\
\!\!\!\!\!\!\!_{a\!^{\star}_{_{i(W_{_\sigma})}}})\equiv{\rho}({\kappa}\!^{^{_{V_{_\pi}}}}\
\!\!\!\!\!\!\!\!_{a\!^{\star}_{_{i(W_{_\pi})}}})
$}

\endpspicture

\bigskip\bigskip\bigskip\bigskip\bigskip\bigskip\bigskip
\bigskip\bigskip\bigskip\bigskip\bigskip\bigskip\bigskip\bigskip\bigskip\bigskip\bigskip\bigskip
\begin{center}
Figure 17
\end{center}
\par\bigskip\smallskip\par
\noindent

\begin{lemma}\label{a24}{\rm Let $({\eta}_{_{\beta^\star}})_{_{\beta^\star\in {B}^\star}}$ be a $\tau_{_q}$-conet
in $(\overline{X},\overline{\mathcal{U}})$ without end point such that
for each $\delta^{\star}>\beta^{\star}$, ${{\mathcal{A}}}_{_{\eta_{_{\delta^{\star}}}}}\nsubseteq {{\mathcal{A}}}_{_{\eta_{_{\beta^{\star}}}}}$.
Then, there is a $\tau_{_q}$-conet
$(s_{_\mu})_{_{\mu\in M}}$ of $(X,{\mathcal{U}})$ such that
the nets $({\eta}_{_{\beta^\star}})_{_{\beta^\star\in {B}^\star}}$ and
$(\phi(s_{_\mu})){_{_{\mu\in M}}}$ are right cofinal.}
\end{lemma}

\par
\smallskip
\par\noindent
We state now the general case.
\par\noindent

\begin{theorem} \label{a25}{\rm For every $\tau$-cut $\xi^{\star}$ in $(\overline{X},\overline{\mathcal{U}})$, there is a
$\tau$-cut $\xi$ in $(X,{\mathcal{U}})$ such that each $\tau_{_p}$-net
of ${{\mathcal{A}}}_{_{\xi^{\star}}}$ $\tau((\overline{\mathcal{U}}))$-converges to $\xi$ and each
$\tau_{_q}$-conet of ${{\mathcal{B}}}_{_{\xi^{\star}}}$ $\tau(\overline{{\mathcal{U}}^{-1}})$-converges to $\xi$.}
\end{theorem}
\begin{proof}
Let $\xi^{\star}=({{\mathcal{A}}}_{_{\xi^{\star}}},{{\mathcal{B}}}_{_{\xi^{\star}}})$
be a $\tau$-cut in $(\overline{X},\overline{\mathcal{U}})$
such that
${{\mathcal{A}}}_{_{\xi^{\star}}}=\{(\xi^{^i}_{_k})_{_{k\in K_i}}\vert
i\in I\}$ and
 ${{\mathcal{B}}}_{_{\xi^{\star}}}=\{(\eta^{^j}_{_\lambda})_{_{\lambda\in \Lambda_j}}\vert j\in J\}$.
We define
\par
${{\mathcal{A}}}_{_{\xi}}=\{(x_{_a})_{_{a\in A}}$ $\vert
(x_{_a})_{_{a\in A}}$ $\tau_{_p}$-net in $X$ and $(\phi(x_{_a})_{_a\in
A}\in {{\mathcal{A}}}_{_{\xi^{\star}}}\ \}$ and
\par
${{\mathcal{B}}}_{_{\xi}}=\{(y_{_\beta})_{_{\beta\in B}}$ $\vert
(y_{_\beta})_{_{\beta\in B}}$ $\tau_{_q}$-conet in $X$ and
$(\phi(y_{_\beta})_{_\beta\in B}\in {{\mathcal{B}}}_{_{\xi^{\star}}}\
\}$ (see fig. 18).

\pspicture(0,0)(0,0)

\psline[linestyle=dotted]{->}(-.2,-2.43)(5.45,-2.42)

\psline[linestyle=dotted]{<-}(5.72,-2.43)(11.31,-2.42)

\psline[linestyle=dotted]{->}(-.2,-4.93)(5.45,-4.92)

\psline[linestyle=dotted]{<-}(5.72,-4.93)(11.31,-4.92)

\Rput[ul](-.6,-2.15){$(x_{_a})_{_{a\in A}}$}

\Rput[ul](-.6,-4.65){$(x_{_\gamma})_{_{\gamma\in \Gamma}}$}

\Rput[ul](10.3,-2.15){$(y_{_\beta})_{_{\beta\in B}}$}

\Rput[ul](10.3,-4.65){$(x_{_\delta})_{_{\delta\in \Delta}}$}

\qdisk(.76,-2.43){2pt}

\qdisk(1,-2.43){2pt}

\qdisk(1.28,-2.43){2pt}

\qdisk(2.16,-2.43){2pt}

\qdisk(2.4,-2.43){2pt}

\qdisk(2.68,-2.43){2pt}

\qdisk(3.56,-2.43){2pt}

\qdisk(3.8,-2.43){2pt}

\qdisk(4.08,-2.43){2pt}

\qdisk(7.06,-2.43){2pt}

\qdisk(7.3,-2.43){2pt}

\qdisk(7.58,-2.43){2pt}

\qdisk(8.46,-2.43){2pt}

\qdisk(8.7,-2.43){2pt}

\qdisk(8.98,-2.43){2pt}

\qdisk(9.86,-2.43){2pt}

\qdisk(10.1,-2.43){2pt}

\qdisk(10.38,-2.43){2pt}

\qdisk(.76,-4.93){2pt}

\qdisk(1,-4.93){2pt}

\qdisk(1.28,-4.93){2pt}

\qdisk(2.16,-4.93){2pt}

\qdisk(2.4,-4.93){2pt}

\qdisk(2.68,-4.93){2pt}

\qdisk(3.56,-4.93){2pt}

\qdisk(3.8,-4.93){2pt}

\qdisk(4.08,-4.93){2pt}

\qdisk(7.06,-4.93){2pt}

\qdisk(7.3,-4.93){2pt}

\qdisk(7.58,-4.93){2pt}

\qdisk(8.46,-4.93){2pt}

\qdisk(8.7,-4.93){2pt}

\qdisk(8.98,-4.93){2pt}

\qdisk(9.86,-4.93){2pt}

\qdisk(10.1,-4.93){2pt}

\qdisk(10.38,-4.93){2pt}

\Rput[ul](3,-.4){${\mathcal{A}}_{_{\xi^{\star}}}$}

\Rput[ul](7.3,-.4){${\mathcal{B}}_{_{\xi^{\star}}}$}

\Rput[ul](5.3,-.4){$\xi^{\star}$}

\pscircle[linewidth=2pt](1.03,-1.80){0.2}

\pscircle[linewidth=2pt](2.43,-1.80){0.2}

\pscircle[linewidth=2pt](3.83,-1.80){0.2}

\multiput(1.64,-2.02)(1.46,-1.80){1}{*}

\multiput(1.94,-2.02)(1.46,-1.80){1}{*}

\multiput(1.34,-2.02)(1.46,-1.80){1}{*}

\multiput(3.34,-2.02)(1.46,-1.80){1}{*}

\multiput(3.04,-2.02)(1.46,-1.80){1}{*}

\multiput(2.74,-2.02)(1.46,-1.80){1}{*}

\multiput(4.76,-2.02)(1.46,-1.80){1}{*}

\multiput(4.46,-2.02)(1.46,-1.80){1}{*}

\multiput(4.16,-2.02)(1.46,-1.80){1}{*}

\psline{->}(4.95,-1.80)(5.45,-1.80)

\multiput(-.1,-2.02)(1.46,-1.30){1}{*}

\multiput(.2,-2.02)(1.46,-1.30){1}{*}

\multiput(.5,-2.02)(1.46,-1.30){1}{*}

\psline{<-}(5.65,-1.80)(6.15,-1.80)

\multiput(6.2,-2.02)(1.46,-1.30){1}{*}

\multiput(6.5,-2.02)(1.46,-1.30){1}{*}

\multiput(6.8,-2.02)(1.46,-1.30){1}{*}

\pscircle[linewidth=2pt](7.3,-1.80){0.2}

\multiput(7.6,-2.02)(1.46,-1.30){1}{*}

\multiput(7.9,-2.02)(1.46,-1.30){1}{*}

\multiput(8.2,-2.02)(1.46,-1.30){1}{*}

\pscircle[linewidth=2pt](8.7,-1.80){0.2}

\multiput(9,-2.02)(1.46,-1.30){1}{*}

\multiput(9.3,-2.02)(1.46,-1.30){1}{*}

\multiput(9.6,-2.02)(1.46,-1.30){1}{*}

\pscircle[linewidth=2pt](10.1,-1.80){0.2}

\multiput(10.4,-2.02)(1.46,-1.30){1}{*}

\multiput(10.7,-2.02)(1.46,-1.30){1}{*}

\multiput(11,-2.02)(1.46,-1.30){1}{*}

\psline[linestyle=dotted]{-}(.91,-1.93)(.62,-2.9)
\psline[linestyle=dotted]{-}(.91,-1.87)(.62,-.79)
\psline[linestyle=dotted]{-}(1.01,-1.93)(1,-2.9)
\psline[linestyle=dotted]{-}(1.01,-1.67)(1,-.79)
\psline[linestyle=dotted]{-}(1.15,-1.93)(1.4,-2.9)
\psline[linestyle=dotted]{-}(1.15,-1.87)(1.4,-.79)

\psline[linestyle=dotted]{-}(2.31,-1.93)(2.02,-2.9)
\psline[linestyle=dotted]{-}(2.31,-1.87)(2.02,-.79)
\psline[linestyle=dotted]{-}(2.41,-1.93)(2.4,-2.9)
\psline[linestyle=dotted]{-}(2.41,-1.67)(2.4,-.79)
\psline[linestyle=dotted]{-}(2.55,-1.93)(2.8,-2.9)
\psline[linestyle=dotted]{-}(2.55,-1.87)(2.8,-.79)

\psline[linestyle=dotted]{-}(3.71,-1.93)(3.42,-2.9)
\psline[linestyle=dotted]{-}(3.71,-1.87)(3.42,-.79)
\psline[linestyle=dotted]{-}(3.81,-1.93)(3.8,-2.9)
\psline[linestyle=dotted]{-}(3.81,-1.67)(3.8,-.79)
\psline[linestyle=dotted]{-}(3.95,-1.93)(4.2,-2.9)
\psline[linestyle=dotted]{-}(3.95,-1.87)(4.2,-.79)

\psline[linestyle=dotted]{-}(7.21,-1.93)(6.92,-2.9)
\psline[linestyle=dotted]{-}(7.21,-1.67)(6.92,-.79)
\psline[linestyle=dotted]{-}(7.31,-1.93)(7.3,-2.9)
\psline[linestyle=dotted]{-}(7.31,-1.67)(7.3,-.79)
\psline[linestyle=dotted]{-}(7.45,-1.93)(7.7,-2.9)
\psline[linestyle=dotted]{-}(7.45,-1.87)(7.7,-.79)

\psline[linestyle=dotted]{-}(8.61,-1.93)(8.32,-2.9)
\psline[linestyle=dotted]{-}(8.61,-1.67)(8.32,-.79)
\psline[linestyle=dotted]{-}(8.71,-1.93)(8.7,-2.9)
\psline[linestyle=dotted]{-}(8.71,-1.67)(8.7,-.79)
\psline[linestyle=dotted]{-}(8.85,-1.93)(9.1,-2.9)
\psline[linestyle=dotted]{-}(8.85,-1.87)(9.1,-.79)

\psline[linestyle=dotted]{-}(10.01,-1.93)(9.72,-2.9)
\psline[linestyle=dotted]{-}(10.01,-1.67)(9.72,-.79)
\psline[linestyle=dotted]{-}(10.11,-1.93)(10.1,-2.9)
\psline[linestyle=dotted]{-}(10.11,-1.67)(10.1,-.79)
\psline[linestyle=dotted]{-}(10.25,-1.93)(10.5,-2.9)
\psline[linestyle=dotted]{-}(10.25,-1.87)(10.5,-.79)

\pscircle[linewidth=2pt](1.03,-4.30){0.2}

\pscircle[linewidth=2pt](2.43,-4.30){0.2}

\pscircle[linewidth=2pt](3.83,-4.30){0.2}

\multiput(1.64,-4.52)(1.46,-1.80){1}{*}

\multiput(1.94,-4.52)(1.46,-1.80){1}{*}

\multiput(1.34,-4.52)(1.46,-1.80){1}{*}

\multiput(3.34,-4.52)(1.46,-1.80){1}{*}

\multiput(3.04,-4.52)(1.46,-1.80){1}{*}

\multiput(2.74,-4.52)(1.46,-1.80){1}{*}

\multiput(4.76,-4.52)(1.46,-1.80){1}{*}

\multiput(4.46,-4.52)(1.46,-1.80){1}{*}

\multiput(4.16,-4.52)(1.46,-1.80){1}{*}

\psline{->}(4.95,-4.30)(5.45,-4.30)

\multiput(-.1,-4.52)(1.46,-1.30){1}{*}

\multiput(.2,-4.52)(1.46,-1.30){1}{*}

\multiput(.5,-4.52)(1.46,-1.30){1}{*}

\psline{<-}(5.65,-4.30)(6.15,-4.30)

\multiput(6.2,-4.52)(1.46,-1.30){1}{*}

\multiput(6.5,-4.52)(1.46,-1.30){1}{*}

\multiput(6.8,-4.52)(1.46,-1.30){1}{*}

\pscircle[linewidth=2pt](7.3,-4.30){0.2}

\multiput(7.6,-4.52)(1.46,-1.30){1}{*}

\multiput(7.9,-4.52)(1.46,-1.30){1}{*}

\multiput(8.2,-4.52)(1.46,-1.30){1}{*}

\pscircle[linewidth=2pt](8.7,-4.30){0.2}

\multiput(9,-4.52)(1.46,-1.30){1}{*}

\multiput(9.3,-4.52)(1.46,-1.30){1}{*}

\multiput(9.6,-4.52)(1.46,-1.30){1}{*}

\pscircle[linewidth=2pt](10.1,-4.30){0.2}

\multiput(10.4,-4.52)(1.46,-1.30){1}{*}

\multiput(10.7,-4.52)(1.46,-1.30){1}{*}

\multiput(11,-4.52)(1.46,-1.30){1}{*}

\psline[linestyle=dotted]{-}(.91,-4.43)(.62,-5.4)
\psline[linestyle=dotted]{-}(.91,-4.37)(.62,-3.29)
\psline[linestyle=dotted]{-}(1.01,-4.43)(1,-5.4)
\psline[linestyle=dotted]{-}(1.01,-4.17)(1,-3.29)
\psline[linestyle=dotted]{-}(1.15,-4.43)(1.4,-5.4)
\psline[linestyle=dotted]{-}(1.15,-4.37)(1.4,-3.29)

\psline[linestyle=dotted]{-}(2.31,-4.43)(2.02,-5.4)
\psline[linestyle=dotted]{-}(2.31,-4.37)(2.02,-3.29)
\psline[linestyle=dotted]{-}(2.41,-4.43)(2.4,-5.4)
\psline[linestyle=dotted]{-}(2.41,-4.17)(2.4,-3.29)
\psline[linestyle=dotted]{-}(2.55,-4.43)(2.8,-5.4)
\psline[linestyle=dotted]{-}(2.55,-4.37)(2.8,-3.29)

\psline[linestyle=dotted]{-}(3.71,-4.43)(3.42,-5.4)
\psline[linestyle=dotted]{-}(3.71,-4.37)(3.42,-3.29)

\psline[linestyle=dotted]{-}(3.81,-4.43)(3.8,-5.4)
\psline[linestyle=dotted]{-}(3.81,-4.17)(3.8,-3.29)

\psline[linestyle=dotted]{-}(3.95,-4.43)(4.2,-5.4)
\psline[linestyle=dotted]{-}(3.95,-4.37)(4.2,-3.29)

\psline[linestyle=dotted]{-}(7.21,-4.43)(6.92,-5.4)
\psline[linestyle=dotted]{-}(7.21,-4.17)(6.92,-3.29)
\psline[linestyle=dotted]{-}(7.31,-4.43)(7.3,-5.4)
\psline[linestyle=dotted]{-}(7.31,-4.17)(7.3,-3.29)
\psline[linestyle=dotted]{-}(7.45,-4.43)(7.7,-5.4)
\psline[linestyle=dotted]{-}(7.45,-4.37)(7.7,-3.29)

\psline[linestyle=dotted]{-}(8.61,-4.43)(8.32,-5.4)
\psline[linestyle=dotted]{-}(8.61,-4.17)(8.32,-3.29)
\psline[linestyle=dotted]{-}(8.71,-4.43)(8.7,-5.4)
\psline[linestyle=dotted]{-}(8.71,-4.17)(8.7,-3.29)
\psline[linestyle=dotted]{-}(8.85,-4.43)(9.1,-5.4)
\psline[linestyle=dotted]{-}(8.85,-4.37)(9.1,-3.29)

\psline[linestyle=dotted]{-}(10.01,-4.43)(9.72,-5.4)
\psline[linestyle=dotted]{-}(10.01,-4.17)(9.72,-3.29)
\psline[linestyle=dotted]{-}(10.11,-4.43)(10.1,-5.4)
\psline[linestyle=dotted]{-}(10.11,-4.17)(10.1,-3.29)
\psline[linestyle=dotted]{-}(10.25,-4.43)(10.5,-5.4)
\psline[linestyle=dotted]{-}(10.25,-4.37)(10.5,-3.29)

\psline{->}(5.58,-5.5)(5.58,-.7)

\psline[linestyle=dotted]{->}(-.2,-7.43)(5.45,-7.42)

\psline[linestyle=dotted]{<-}(5.72,-7.43)(11.31,-7.42)

\psline[linestyle=dotted]{->}(-.2,-8.43)(5.45,-8.42)

\psline[linestyle=dotted]{<-}(5.72,-8.43)(11.31,-8.42)

\qdisk(.76,-7.43){2pt}

\qdisk(1,-7.43){2pt}

\qdisk(1.28,-7.43){2pt}

\qdisk(2.16,-7.43){2pt}

\qdisk(2.4,-7.43){2pt}

\qdisk(2.68,-7.43){2pt}

\qdisk(3.56,-7.43){2pt}

\qdisk(3.8,-7.43){2pt}

\qdisk(4.08,-7.43){2pt}

\qdisk(7.06,-7.43){2pt}

\qdisk(7.3,-7.43){2pt}

\qdisk(7.58,-7.43){2pt}

\qdisk(8.46,-7.43){2pt}

\qdisk(8.7,-7.43){2pt}

\qdisk(8.98,-7.43){2pt}

\qdisk(9.86,-7.43){2pt}

\qdisk(10.1,-7.43){2pt}

\qdisk(10.38,-7.43){2pt}

\qdisk(.76,-8.43){2pt}

\qdisk(1,-8.43){2pt}

\qdisk(1.28,-8.43){2pt}

\qdisk(2.16,-8.43){2pt}

\qdisk(2.4,-8.43){2pt}

\qdisk(2.68,-8.43){2pt}

\qdisk(3.56,-8.43){2pt}

\qdisk(3.8,-8.43){2pt}

\qdisk(4.08,-8.43){2pt}

\qdisk(7.06,-8.43){2pt}

\qdisk(7.3,-8.43){2pt}

\qdisk(7.58,-8.43){2pt}

\qdisk(8.46,-8.43){2pt}

\qdisk(8.7,-8.43){2pt}

\qdisk(8.98,-8.43){2pt}

\qdisk(9.86,-8.43){2pt}

\qdisk(10.1,-8.43){2pt}

\qdisk(10.38,-8.43){2pt}

\Rput[ul](-.6,-7.15){$(x_{_a})_{_{a\in A}}$}

\Rput[ul](-.6,-8.15){$(x_{_\gamma})_{_{\gamma\in \Gamma}}$}

\Rput[ul](10.3,-7.15){$(y_{_\beta})_{_{\beta\in B}}$}

\Rput[ul](10.3,-8.15){$(x_{_\delta})_{_{\delta\in \Delta}}$}

\Rput[ul](5.35,-6.5){$\xi$}

\Rput[ul](2.35,-6.9){${\mathcal{A}}_{_{\xi}}$}

\Rput[ul](8,-6.9){${\mathcal{B}}_{_{\xi}}$}

\psline{<-}(5.55,-6.8)(5.55,-8.5)

\endpspicture
\par
\bigskip\bigskip\bigskip\bigskip\bigskip\bigskip\bigskip\bigskip\bigskip
\bigskip\bigskip\bigskip\bigskip\bigskip\bigskip\bigskip\bigskip\bigskip\bigskip\bigskip\bigskip\bigskip\bigskip
\bigskip\bigskip\bigskip
\bigskip\bigskip
\begin{center}
Figure 18
\end{center}
\par\bigskip\smallskip\par
\noindent

We first verify that the pair $\xi=({{\mathcal{A}}}_{_{\xi}},{{\mathcal{B}}}_{_{\xi}})$ constitutes a
$\tau$-cut in $(X,{\mathcal{U}})$. It is need to be proved:
\par
\smallskip
\par\noindent
(A) {\it The classes ${{\mathcal{A}}}_{_{\xi}}$, ${{\mathcal{B}}}_{_{\xi}}$ are non-void} and
\par\noindent
(B) {\it The pair $({{\mathcal{A}}}_{_{\xi}},{{\mathcal{B}}}_{_{\xi}})$ satisfies the two conditions of the Definition
\ref{a3}.}
\par
\smallskip\bigskip
\par\noindent
For (A): we consider $(\xi^{^i}_{_k})_{_{k\in K_{_i}}}\in {\mathcal{A}}_{_{\xi^{\star}}}$. We distinguish two cases.
\par
\smallskip
\par\noindent
{\it {\rm (a)}. The $\tau_{_p}$-net $(\xi^{^i}_{_k})_{_{k\in K_i}}$ is finally
constant}.
\par
\smallskip
\par

In this case, there exists $k_{_0}\in K_{_i}$ such that
$\xi^{^i}_{_k}=\xi^{^i}_{_{k_{_0}}}$ for every $k\geq k_{_0}$. Suppose that
$(x_{_a})_{_{a\in A}}\in {{\mathcal{A}}}_{_{\xi^{^i}_{_{k_{_0}}}}}$. Then,
Proposition \ref{a18} implies that
$(\xi^{^i}_{_{k_{_0}}},(\phi(x_{_a}))_a)\longrightarrow 0$ which jointly to
$(\eta^{^j}_{_\lambda},\xi^{^i}_{_{k_{_0}}})\longrightarrow 0$ we conclude that
$(\eta^{^j}_{_\lambda},(\phi(x_{_a}))_a)\longrightarrow 0$. Thus
$(x_{_a})_{_{a\in A}}\in {{\mathcal{A}}}_{_{\xi^{\star}}}$.
\par
\smallskip
\par\noindent
{\it {\rm (b)}. The $\tau_{_p}$-net $(\xi^{^i}_{_k})_{_{k\in K_i}}$ is not
finally constant}.
\par\noindent
We distinguish two subcases.
\par
\smallskip
\par\noindent
{\rm ($b_1$)} {\it There is a $k_{_0}\in K_{_i}$ such that for each
$k^{\prime}>k\geq k_{_0}$, ${{\mathcal{A}}}_{_{\xi^{^i}_{_k}}}\subseteq {{\mathcal{A}}}_{_{\xi^{^i}_{_{k^{\prime}}}}}$}.
\par\noindent
Suppose that $(x_a)_{_{a\in A}}\in
{{\mathcal{A}}}_{_{\xi^{^i}_{_{k_{_0}}}}}$. Then, from
$((\eta^{^j}_{_\lambda})_{_{\lambda}},(\xi^{^i}_{_{k}})_{_k})\longrightarrow
0$, $((\xi^{^i}_{_k})_{_k},\xi^{^i}_{_{k_{_0}}})\longrightarrow 0$ and
$(\xi^{^i}_{_{k_{_0}}},(\phi(x_a))_a)\longrightarrow 0$ we conclude that
$((\eta^{^j}_{_\lambda})_{_\lambda},(\phi(x_a))_a)\longrightarrow 0$
and thus $(x_{_a})_{_{a\in A}}\in {{\mathcal{A}}}_{_{\xi}}$.
\par
\smallskip
\par\noindent

{\rm ($b_2$)} {\it For each $k\in K$ there is a $k^{\prime}>k$ such
that ${{\mathcal{A}}}_{_{\xi^{^i}_{_k}}}\nsubseteq {{\mathcal{A}}}_{_{\xi^{^i}_{_{k^{\prime}}}}}$}.
\par\noindent
In this
case we can find a $\tau$-subnet
$(\xi^{^i}_{_{k_{_\mu}}})_{_{\mu\in M}}$ of
$(\xi^{^i}_{_k})_{_{k\in K_{_i}}}$ such that for each $\mu^{\prime}>\mu$, it is
${{\mathcal{A}}}_{_{\xi^{^i}_{_{k_{_\mu}}}}}
\nsubseteq {{\mathcal{A}}}_{_{\xi^{^i}_{_{k_{\mu^{\prime}}}}}}$.
According to Lemma \ref{a21}, there exists a $\tau_{_p}$-net $(x_a)_{_{a\in
A}}$ in $(X,{\mathcal{U}})$ whose the $\phi$-images constitutes a
$\tau_{_p}$-net left cofinal to
$(\xi^{^i}_{_{k_{_\mu}}})_{_{\mu\in M}}$
and finally
left cofinal to $(\xi^{^i}_{_k})_{_{k\in K_{_i}}}$ (Proposition
\ref{a11}). Since $(\xi^{^i}_{_k})_{_{k\in K_{_i}}}\in {\mathcal{A}}_{_{\xi^{\star}}}$, Proposition \ref{a12} implies that
$(\phi(x_a)){_{_{a\in A}}}\in {\mathcal{A}}_{_{\xi^{\star}}}$. Hence,
$(x_a){_{_{a\in A}}}\in {\mathcal{A}}_{_{\xi}}$. Thus ${{\mathcal{A}}}_{_{\xi}}$ is non-void. Similarly it is proved that ${{\mathcal{B}}}_{_{\xi}}$ is non-void.
\par
\bigskip
\par
For (B): we firstly prove that the members of ${{\mathcal{A}}}_{_{\xi}}$ has as $\tau_{_q}$-conets all the members of ${{\mathcal{B}}}_{_{\xi}}$. Indeed, by the construction of ${\mathcal{A}}_{_{\xi}}$ and ${\mathcal{B}}_{_{\xi}}$ we have that
$((\phi(y_{_\beta}))_{_\beta},(\phi(x_a))_a)\longrightarrow 0$.
But then, Theorem \ref{a17} implies that
$((y_{_\beta})_{_\beta},(x_a)_a)\longrightarrow 0$. So, the first
demand for the being $\xi$ a $\tau$-cut is fulfilled.

For the second demand, let us assume that
$(t_{_\gamma})_{_{\gamma\in \Gamma}}$ is a $\tau_{_p}$-net which has as
$\tau_{_q}$-conets all the members of ${{\mathcal{B}}}_{_{\xi}}$. We
have to prove that $(t_{_\gamma})_{_{\gamma\in \Gamma}}\in {{\mathcal{A}}}_{_{\xi}}$, that is,
$((\eta^{^j}_{_\lambda})_{_\lambda},(\phi(t_{_{\gamma}}))_{_\gamma})\longrightarrow
0$ for each $j\in J$. Fix a $j\in J$. We distinguish two cases.
\par
\bigskip
\par\noindent
{\it ${\rm (a^{\prime})}$. The $\tau_{_q}$-conet
$(\eta^{^j}_{_\lambda})_{_{\lambda\in \Lambda_j}}\in {{\mathcal{B}}}_{_{\xi^{\star}}}$ is constant}.
\par\noindent
Suppose that $\eta^{^j}_{_\lambda}=\eta^{^j}_{_{\lambda_{_0}}}$ for
each $\lambda\geq \lambda_{_0}$. If $(\xi^{^i}_{_k})_{_{k\in K_i}}\in
{{\mathcal{A}}}_{_{\xi^{\star}}}$, then
$(\eta^{^j}_{_{\lambda_{_0}}},(\xi^{^i}_{_k})_{_k})\longrightarrow
0$. On the other hand, if $(y_{_\beta})_{_{\beta\in B}}\in {{\mathcal{B}}}_{_{\eta^{^j}_{_{\lambda_{_0}}}}}$, then
$((\phi(y_{_\beta}))_{_\beta},\eta^{^j}_{_{\lambda_{_0}}})\longrightarrow
0$ and by
$(\xi^{^i}_{_k})_{_{k\in K}}\in
{{\mathcal{A}}}_{_{\xi^{\star}}}$ we conclude that $((\phi(y_{_\beta}))_{_\beta},
(\xi^{^i}_{_k})_{_k})\longrightarrow 0$.
Therefore,
$(y_{_\beta})_{_{\beta\in B}}\in {{\mathcal{B}}}_{_{\xi}}$ which implies that
${{\mathcal{B}}}_{_{\eta^{^j}_{_{\lambda_{_0}}}}}\subseteq  {{\mathcal{B}}}_{_{\xi}}$.
But then, $(t_{_\gamma})_{_{\gamma\in \Gamma}}$
has as
$\tau_{_q}$-conets all the members of
${{\mathcal{B}}}_{_{\eta^{^j}_{_{\lambda_{_0}}}}}$ which concludes that
$(t_{_\gamma})_{_{\gamma\in \Gamma}}\in {{\mathcal{A}}}_{_{\eta^{^j}_{_{\lambda_{_0}}}}}= {{\mathcal{A}}}_{_{\eta^{^j}_{_{\lambda}}}}$.
The last conclusion implies that
that
$((\eta^{^j}_{_{\lambda}})_{_\lambda},(\phi(t_{_{\gamma}}))_{_\gamma})\longrightarrow
0$, as it is desired.
\par
\bigskip
\par\noindent
{\it ${\rm (b^{\prime})}$. The $\tau_{_q}$-conet
$(\eta^{^j}_{_\lambda})_{_{\lambda\in \Lambda_j}}$ is non-constant}.
\par\noindent
We distinguish two subcases.
\par
\bigskip
\par\noindent
${\rm (b_1^{\prime})}$. {\it There is a $\lambda_{_0}\in \Lambda$
such that for each $\lambda^{\prime}>\lambda\geq \lambda_{_0}$,
${{\mathcal{A}}}_{_{{\eta^{^j}_{_{\lambda^{\prime}}}}}}\subseteq {{\mathcal{A}}}_{_{\eta^{^j}_{_\lambda}}}$}.
\par\noindent
In this case, we have ${{\mathcal{B}}}_{_{\eta^{^j}_{_\lambda}}}\subseteq {{\mathcal{B}}}_{_{{\eta^{^j}_{_{\lambda^{\prime}}}}}}$ and for a fixed
$\lambda\geq\lambda_{_0}$ and a $(y_{_\beta})_{_{\beta\in B}}\in
{{\mathcal{B}}}_{_{\eta^{^j}_{_{\lambda}}}}$ we have
$(y_{_\beta})_{_{\beta\in B}}\in {{\mathcal{B}}}_{_{{\eta^{^j}_{_{\lambda^{\prime}}}}}}$ for each
$\lambda^{\prime}\geq\lambda$. From
$((\phi(y_{_\beta}))_{_\beta},(\eta^{^j}_{_{\lambda^{\prime}}})_{_{\lambda^{\prime}}})\longrightarrow
0$ and
$((\eta^{^j}_{_{\lambda^{\prime}}})_{_{\lambda^{\prime}}},(\xi^{^i}_{_k})_{_k})\longrightarrow
0$, we conclude that
$((\phi(y_{_\beta}))_{_\beta},(\xi^{^i}_{_k})_{_k})\longrightarrow
0$. Hence,
$(y_{_\beta})_{_{\beta\in B}}\in {{\mathcal{B}}}_{_{\xi^{\star}}}$.
As in the proof for the case ${\rm (a^{\prime})}$, we conclude that
$((\eta^{^j}_{_{\lambda}})_{_\lambda},(\phi(t_{_{\gamma}}))_{_\gamma})\longrightarrow
0$.

\smallskip
\par\noindent
${\rm (b_{_2}^{\prime})}$. {\it For each $\lambda\in \Lambda$ there
is $\lambda^{\prime}>\lambda$ such that ${{\mathcal{A}}}_{_{{\eta^{^j}_{_{\lambda^{\prime}}}}}}\nsubseteq {{\mathcal{A}}}_{_{\eta^{^j}_{_\lambda}}}$}.
\par\noindent
In this case, for each $j\in J$, we can find a $\tau$-subnet
$({\eta}^{^j}_{_{\lambda_{_\sigma}}})_{_{\sigma\in \Sigma}}$ such
that for every $\sigma^{\prime}>\sigma$ there holds ${{\mathcal{A}}}_{_{{\eta}^{^j}_{_{\lambda_{_{\sigma^{\prime}}}}}}}\nsubseteq
{{\mathcal{A}}}_{_{{\eta}^{^j}_{_{\lambda_{_{\sigma}}}}}}$. Then, from the
Lemma \ref{a24}, there exists a $\tau_{_q}$-conet $(y_{_\beta})_{_{\beta\in
B}}$ in $(X,{\mathcal{U}})$ such that the $\tau_{_q}$-conets
$({\eta}^{^j}_{_{\lambda_{_\sigma}}})_{_{\sigma\in \Sigma}}$ and
$(\phi(y_{_\beta})){_{_{\beta\in B}}}$ are right cofinal. Since
$({\eta}^{^j}_{_{\lambda}})_{_{\lambda\in \Lambda}}\in {\mathcal{B}}_{_{\xi^{\star}}}$, Proposition \ref{a12} implies that
$(\phi(y_{_\beta})){_{_{\beta\in B}}}\in {\mathcal{B}}_{_{\xi^{\star}}}$. Hence,
$(y_{_\beta}){_{_{\beta\in B}}}\in {\mathcal{B}}_{_{\xi}}$.
Therefore,
$((y_{_\beta})_{_\beta},(t_{_{\gamma}})_{_\gamma})\longrightarrow 0$
or equivalently
$((\phi(y_{_\beta}))_{_\beta},(\phi(t_{_{\gamma}}))_{_\gamma})\longrightarrow
0$. But then, since $({\eta}^{^j}_{_{\lambda_{_\sigma}}})_{_{\sigma\in \Sigma}}$ and
$(\phi(y_{_\beta})){_{_{\beta\in B}}}$ are right cofinal we conclude that
$(({\eta}^{^j}_{_\lambda})_{_\lambda},(\phi(t_{_{\gamma}}))_{_\gamma})\longrightarrow
0$.
Hence, in any case we have that $(({\eta}^{^j}_{_\lambda})_{_\lambda},(\phi(t_{_{\gamma}}))_{_\gamma})\longrightarrow
0$ which implies that $(t_{_\gamma})_{_{\gamma\in \Gamma}}\in {{\mathcal{A}}}_{_{\xi}}$.
\par
\par
\smallskip
\par
Likewise, we prove that if a $\tau_{_p}$-net has as $\tau_{_q}$-conets all the
elements of ${{\mathcal{B}}}_{_{\xi}}$, then it belongs to ${{\mathcal{A}}}_{_{\xi}}$. Thus, $\xi$ constitute a $\tau$-cut
in $(X,{\mathcal{U}})$.
\par
\smallskip
\par\noindent
It remains to prove that for every $i\in I$ (resp. for every $j\in
J$), $(\xi^{^i}_{_k})_{_{k\in
K_i}}$ (resp. $(\eta^{^j}_{_\lambda})_{_{\lambda\in \Lambda_j}}$)
is $\tau(\overline{\mathcal{U}})$-convergent (resp. $\tau((\overline{\mathcal{U}})^{-1})$-convergent) to $\xi$.
It is enough to prove it for $(\xi^{^i}_{_k})_{_{k\in K_i}}$.
\par
We consider a $\tau_{_p}$-net $(\xi^{^i}_{_k})_{_{k\in K}}\in {\mathcal{A}}_{_{\xi^{\star}}}$ and the $\tau$-cut $\xi$ in $X$ the
constructed above. If $(\xi^{^i}_{_k})_{_{k\in K_i}}$ is finally
constant, then there exists $k_{_0}\in K$ such that
$\xi^{^i}_{_k}=\xi^{^i}_{_{k_{_0}}}$ for every $k>k_{_0}$. But then,
as in the case (a) of (A), we have ${{\mathcal{A}}}_{_{\xi^{^i}_{_{k}}}}={{\mathcal{A}}}_{_{\xi^{^i}_{_{k_{_0}}}}}\subseteq {{\mathcal{A}}}_{_{\xi}}$ and thus
$(\xi,\xi^{^i}_{_k})\longrightarrow 0$.
\par
If $(\xi^{^i}_{_k})_{_{k\in K}}$ is not finally constant, we
have to examine the cases ($b_{1}$) and ($b_{2}$) of (A).
In the first one, there exists $k_{_0}\in K$ such that for each
$k>k_{_0}$, ${{\mathcal{A}}}_{_{\xi^{^i}_{_k}}}\subseteq {{\mathcal{A}}}_{_{\xi^{^i}_{_{k_{_0}}}}}\subseteq {{\mathcal{A}}}_{_{\xi}}$. Hence,
$(\xi,\xi^{^i}_{_k})\longrightarrow 0$. In the
($b_{2}$)-case, we can extract a subnet
$({\xi}^{^i}_{_{k_{_\sigma}}})_{_{\sigma\in \Sigma}}$ of
$(\xi^{^i}_{_k})_{_{k\in K_i}}$ and a $\tau_{_p}$-net $(x_{_a})_{_{a\in
A}}\in {\mathcal{A}}_{_{\xi}}$ such that
$({\xi}^{^i}_{_{k_{_\sigma}}})_{_{\sigma\in \Sigma}}$ and
$(\phi(x_{_a})){_{_{a\in A}}}$ are left cofinal. By the Proposition
\ref{a18}, $(\phi(x_{_a})){_{_{a\in A}}}$ $\tau(\overline{\mathcal{U}})$-convergence to $\xi$. Therefore, by using the
Proposition \ref{a13}, we conclude that
$({\xi}^{^i}_{_{k_{_\sigma}}})_{_{\sigma\in \Sigma}}$ $\tau(\overline{\mathcal{U}})$-convergence to $\xi$. Finally, Proposition \ref{a11}
and Proposition \ref{a13} imply that
$(\xi,\xi^{^i}_{_k})\longrightarrow 0$. This completes the
proof.
\end{proof}
\par
\smallskip
\par
The previous theorem implies the following theorem.
\par

\par\noindent

\begin{theorem} \label{a26}{\rm Every $T_{_0}$ quasi-uniform space has a
$\tau$-completion.}
\end{theorem}

\begin{lemma}\label{a27} {\rm If $(X,{\mathcal{U}})$ is $T_{_0}$, so is $(\overline{X},\overline{\mathcal{U}})$.}

\end{lemma}

\begin{proof}Let $(\xi,\xi^{\star})\in \bigcap\{ \overline{U}\cap
({\overline{U}})^{-1}\vert U\in {\mathcal{U}}_{_0} \}$. Suppose that
$U\in {\mathcal{U}}_{_0}$. There is a net $(x_a)_{a\in A}\in {\mathcal{A}}_{_\xi}$ such that for each $(x_{_{\beta}}\!\!^i)_{_{\beta\in
B_i}}\in {\mathcal{A}}_{_{\xi^{\star}}}$ there holds
$\tau.((x_a)_a,(x_{_{\beta}}\!\!^i)_{_\beta})\in U$. By definition
there is a $W\in {\mathcal{U}}_{_0}$ such that $W^{-1}(x_a)\times
W(x_{_{\beta}}\!\!^i)\subseteq U$. Thus for each
$(y_{_a}^{^j})_{_{a\in A_j}}\in {\mathcal{B}}_{_\xi}$ we have that
$\tau.(y_{_a}^{^j},x_{_{\beta}}\!\!^i)\in U$. Hence ${\mathcal{B}}_{_{\xi}}\subseteq {\mathcal{B}}_{_{\xi^{\star}}}$. Similarly from
 $(\xi,\xi^{\star})\in \bigcap(\overline{U})^{-1}$ we conclude that
${\mathcal{B}}_{_{\xi^{\star}}}\subseteq {\mathcal{B}}_{_{\xi}}$. Thus
$\xi=\xi^{\star}$.
\end{proof}

\par
\smallskip
\par
We recall the classical definition of idempotency (adapted to the
specific case of the $\tau$-completion).

\begin{definition}\label{a28}{\rm The $\tau$-completion of a $T_{_0}$ quasi-uniform space $(X,d)$ is {\it idempotent}
if and only if there exists a quasi-uniform isomorphism
$\overline{\phi}:\overline{X}\longrightarrow
\overline{\overline{X}}$ such that for each $\xi^{\star}\in
\overline{X}$, we have
$\overline{\phi}(\xi^{\star})=\overline{\xi^{\star\star}}$
($\xi^{\star\star}\in \overline{\overline{X}}$).}
\end{definition}

The {\it idempotency} of the $\tau$-completion is established by the

\begin{theorem}\label{a29}{\rm The $\tau$-completion of a $T_{_0}$ quasi-uniform space $(X,{\mathcal{U}})$ is idempotent.}
\end{theorem}
\begin{proof}We consider a $\tau$-cut
$\xi^{\star\star}\in \overline{\overline{X}}$ where
$\xi^{\star\star}=({\mathcal{A}}_{_{\xi^{\star\star}}},{{\mathcal{B}}}_{_{\xi^{\star\star}}})$. Let $\overline{\phi}$ be the
canonical embedding of $\overline{X}$ into
$\overline{\overline{X}}$.  Then, as in the Theorem \ref{a25}, we define
$\tau$-cut $\xi^{\star}\in \overline{X}$ which is the end point of
$\xi^{\star\star}$. Hence,
$\xi^{\star\star}=\overline{\phi}(\xi^{\star})$ which implies that
$\overline{\overline{X}}=\overline{\phi}(\overline{X})$. On the other hand, by the aim of the Theorems \ref{a25} and \ref{a17} we can check that for each
$U\in \mathcal{U}_{_0}$,
$(\xi^{\star\star},\eta^{\star\star})\in \overline{\overline{U}}\Leftrightarrow
(\overline{\phi}(\xi^{\star}),\overline{\phi}(\eta^{\star}))\in \overline{\overline{U}}\Leftrightarrow
(\xi^{\star},\eta^{\star})\in \overline{U}$.
The rest is obvious.
\end{proof}

\par\bigskip\bigskip\par\noindent
{\it Address}: {\tt {Athanasios Andrikopoulos} \\ {Department of Economics\\ University of Ioannina\\ Greece}
\par\smallskip\par\noindent
{\it E-mail address}:{\tt aandriko@cc.uoi.gr}

\par\bigskip\bigskip\par\noindent
{\it Address}: {\tt {John Stabakis} \\ {Department of Mathematics\\ University of Patras\\ Greece}
\par\smallskip\par\noindent
{\it E-mail address}:{\tt jns@math.upatras.gr}

\end{document}